\documentclass[leqno,final]{siamltex}
\usepackage{graphicx} 
\usepackage{amsmath,amstext,amssymb,bm}
\usepackage{leftidx}

\usepackage{xcolor} 
\usepackage{soul} 
\usepackage{tikz}
\usetikzlibrary{shapes,arrows}
\usepackage{mathrsfs}
\usepackage{epstopdf}
\usepackage{color}
\usepackage{multirow}
\usepackage{tabularx}
\usepackage[shortlabels]{enumitem}

\numberwithin{equation}{section}
\newtheorem{remark}{Remark}[section]

\allowdisplaybreaks[4]

\newcommand{\E}{\mathcal{E}}
\newcommand{\cT}{\mathcal{T}}
\renewcommand{\div}{\mbox{\rm div\,}}

\newcommand{\cF}{\mathcal{F}}

\newcommand{\mP}{\mathbb{P}}
\newcommand{\mH}{\mathbb{H}}
\newcommand{\mE}{\mathbb{E}}
\newcommand{\mV}{\mathbb{V}}

\newcommand{\Ome}{\Omega}
\newcommand{\p}{\partial}
\newcommand{\nab}{\nabla}
\newcommand{\vu}{{\bf u}}
\newcommand{\vB}{{\bf B}}
\newcommand{\vW}{{\bf W}}
\newcommand{\vP}{{\bf P}}
\newcommand{\vf}{{\bf f}}
\newcommand{\vH}{{\bf H}}

\newcommand{\vX}{{\bf X}}
\newcommand{\vL}{{\bf L}}
\newcommand{\vv}{{\bf v}}
\newcommand{\ve}{{\bf e}}
\newcommand{\vE}{{\bf E}}
\newcommand{\vQ}{\mathbf{Q}}

\newcommand{\pphi}{\pmb{\phi}}

\newcommand{\eeta}{\pmb{\eta}}
\newcommand{\vA}{{\bf A}}

\begin{document}
			
		\title{High moment and pathwise error estimates for fully discrete mixed finite element approximations of the Stochastic Stokes Equations with Multiplicative Noises\thanks{This work was partially supported by the NSF grant DMS-1620168.} }
		\markboth{LIET VO}{HIGH MOMENT ERROR ESTIMATES FOR STOCHASTIC STOKES EQUATIONS}

		\author{Liet Vo\thanks{Department of Mathematics, The University of Tennessee, Knoxville, TN 37996, U.S.A. (lvo6@vols.utk.edu).}}

\maketitle

\begin{abstract}
This paper is concerned with high moment and pathwise error estimates for both velocity and pressure approximations of the Euler-Maruyama scheme for time discretization and its two fully discrete mixed finite element discretizations. The main idea for deriving the high moment error estimates for the velocity approximation is to use a bootstrap technique starting from the second moment error estimate. The pathwise error estimate, which is sub-optimal in the energy norm, is obtained by using Kolmogorov's theorem based on the high moment error estimates. Unlike for the velocity error estimate, the higher moment and pathwise error estimates for the pressure approximation are derived in a time-averaged norm. In addition, the impact of noise types on the rates of convergence for both velocity and  pressure approximations is also addressed. 

\end{abstract}





\section{Introduction}\label{sec-1}
In this paper we establish high moment and pathwise error estimates for
fully discrete mixed finite element approximations of the following time-dependent stochastic Stokes problem:
\begin{subequations}\label{eq1.1}
	\begin{alignat}{2} \label{eq1.1a}
	d\vu &=\bigl[\nu\Delta \vu-\nabla p + \vf\bigr] dt + \vB(\vu) d W(t)  &&\qquad\mbox{a.s. in}\, D_T:=(0,T)\times D,\\
	\div \vu &=0 &&\qquad\mbox{a.s. in}\, D_T,\label{eq1.1b}\\
	\vu(0)&= \vu_0 &&\qquad\mbox{a.s. in}\, D,\label{eq1.1d}
	\end{alignat}
\end{subequations}
where $D = (0,L)^d\subset \mathbb{R}^d \, (d=2,3)$ represents a period of the periodic domain in $\mathbb{R}^d$, $\vu$ and $p$ stand for respectively the velocity field and the pressure of the fluid,
$\vB$ is an operator-valued random field, $\{W(t); t\geq 0\}$ denotes an $\mathbb{R}$-valued
Wiener process, and $\vf$ is a body force function (see Section \ref{sec2} for their precise definitions).  Here we seek periodic-in-space solutions $(\vu,p)$ with period $L$, that is, 
$\vu(t,{\bf x} + L{\bf e}_i) = \vu(t,{\bf x})$ and $p(t,{\bf x}+L{\bf e}_i)=p(t,{\bf x})$ 
almost surely 
and for any $(t, {\bf x})\in (0,T)\times \mathbb{R}^d$  and $1\leq i\leq d$, where 
$\{\bf e_i\}_{i=1}^d$ denotes the canonical basis of $\mathbb{R}^d$.

The above stochastic Stokes equations can be viewed as a stochastically perturbation of 
the deterministic non-stationary Stokes equations by a white-noise driven random force
$B(\vu) \frac{d W(t)}{dt}$, it intends to model turbulence flows and also serves as a 
prototypical stochastic partial differential equation (SPDE) model to study analytically 
and to approximate numerically (cf. \cite{Bensoussan95,Chow07,HM06,LRS03,Breit,CHP12,
BM18,Feng1}). 
It should be noted that although the Stokes operator is linear, since $B(\vu)$ is nonlinear 
in $\vu$, the stochastic Stokes system \eqref{eq1.1a} is intrinsically a nonlinear system.

Numerical analysis of the stochastic Stokes (as well as the stochastic Navier-Stokes) equations 
has received a lot of attention in the recent years, various numerical methods, including finite
element and mixed finite element, stabilized methods and splitting methods, have been developed 
and analyzed (cf. \cite{BBM14,BM18,BM21,Breit,CHP12,CP12,Feng,Feng1,Feng2}). 
Optimal and sub-optimal error estimates in strong and weak norms have been established. 
Unlike in the deterministic case,  the primary goal of the numerical analysis of SPDEs is 
to derive error estimates for the quantities of stochastic interests of the error functions. 
The best known such 
quantities are the $p$th moment, $\mathbb{E}[\|u-U\|^p]$ for $2\leq p \leq \infty$ as well as
the variance $\mbox{Var}[\|u-U\|]$, where $\mathbb{E}[\cdot]$ and $\mbox{Var}[\cdot]$ stand for the expectation 
and variance operators, $u$ and $U$ denote respectively the exact and numerical solutions and $\|\cdot\|$ denotes some space-time norm. We note that when $p=\infty$, such an estimate 
is often called a pathwise error estimate. As expected, among these quantities of stochastic interests, the easiest and most sought-after one is the second moment $\mathbb{E}[\|u-U\|^2]$. 
This is exactly what was done in the above cited works for problem \eqref{eq1.1}. 
To the best of our knowledge, no high moment error estimates have been obtained for the stochastic Stokes (and Navier-Stokes) equations in the literature. 
The goal of this paper is to fill this gap by establishing arbitrarily high order moment and pathwise 
error estimates for both velocity and pressure approximations of the stochastic Stokes problem \eqref{eq1.1} discretized by two fully discrete mixed finite element methods. This paper extends the results of \cite{Feng,Feng1} in which the second moment error estimates
were obtained for the mixed finite element methods.

The remainder of this paper is organized as follows. In Section \ref{sec2}, we introduce notations   and preliminaries which include the solution definition and the well-posedness of the stochastic Stokes problem \eqref{eq1.1}. In Section \ref{section_semi}, we first formulate the Euler-Maruyama 
time-stepping scheme for problem \eqref{eq1.1} and then derive high moment and pathwise error estimates for the velocity and pressure approximations of the time-stepping scheme. 
Our main idea for deriving the high moment error estimates for the velocity approximation is to use a bootstrap technique starting from the second moment error estimate and the pathwise error estimate,
which is sub-optimal in the energy norm,
is obtained by using Kolmogorov's theorem based on the high moment error estimates.
In Section \ref{section_fullydiscrete}, the standard mixed finite element method is introduced for the spatial discretization. The stable Taylor-Hood mixed element is chosen as a prototypical example for analysis. The highlight of this section is to derive high moment and pathwise error estimates for the velocity and pressure approximations of the mixed finite element method. Finally, in Section \ref{section_optimal}, we consider the modified mixed method of \cite{Feng1} for problem \eqref{eq1.1} and obtain high moment and pathwise error estimates for this non-standard 
mixed finite method as well.

\section{Preliminaries}\label{sec2}
Standard function and space notation will be adopted in this paper. 
Let $\vH^1_0(D)$ denote the subspace of $\vH^1(D)$ whose ${\mathbb R}^d$-valued functions have zero trace on $\p D$, and $(\cdot,\cdot):=(\cdot,\cdot)_D$ denote the standard $L^2$-inner product, with induced norm $\Vert \cdot \Vert$. We also denote ${\bf L}^p_{per}(D)$ and ${\bf H}^{k}_{per}(D)$ as the Lebesgue and Sobolev spaces of the functions that are periodic and have vanishing mean, respectively. 
Let $(\Omega,\cF, \{\cF_t\},\mP)$ be a filtered probability space with the probability measure $\mP$, the 
$\sigma$-algebra $\cF$ and the continuous  filtration $\{\cF_t\} \subset \cF$. For a random variable $v$ 
defined on $(\Omega,\cF, \{\cF_t\},\mP)$,
${\mathbb E}[v]$ denotes the expected value of $v$. 
For a vector space $X$ with norm $\|\cdot\|_{X}$,  and $1 \leq p < \infty$, we define the Bochner space
$\bigl(L^p(\Omega,X); \|v\|_{L^p(\Omega,X)} \bigr)$, where
$\|v\|_{L^p(\Omega,X)}:=\bigl({\mathbb E} [ \Vert v \Vert_X^p]\bigr)^{\frac1p}$.
We also define 
\begin{align*}
{\mathbb H} := \bigl\{{\bf v}\in  \vL^2_{per}(D) ;\,\div {\bf v}=0 \mbox{ in }D\, \bigr\}\, , \quad 
{\mathbb V} :=\bigl\{{\bf v}\in  \vH^1_{per}(D) ;\,\div {\bf v}=0 \mbox{ in }D \bigr\}\, .
\end{align*}

We recall from \cite{Girault_Raviart86} that the (orthogonal) Helmholtz projection 
${\bf P}_{{\mathbb H}}: \vL^2_{per}(D) \rightarrow {\mathbb H}$ is defined 
by ${\bf P}_{{\mathbb H}} {\bf v} = \pmb{\eta}$ for every ${\bf v} \in \vL^2_{per}(D)$, 
where $(\pmb{\eta}, \xi) \in {\mathbb H} \times H^1_{per}(D)/\mathbb{R}$ is a unique tuple such that 
${\bf v} = \pmb{\eta} + \nabla \xi\, , $
and $\xi\in H^1_{per}(D)/\mathbb{R}$ solves the following Poisson problem 
with the homogeneous Neumann boundary condition:
\begin{equation}\label{poisson}
\Delta \xi = \div {\bf v}.
\end{equation}
We also define the Stokes operator ${\bf A} := -{\bf P}_{\mathbb H} \Delta: {\mathbb V} \cap \vH^2_{per}(D) \rightarrow {\mathbb H}$.

Throughout this paper we assume that  ${\bf B}: \vL^2_{per}(D) \rightarrow \vL^2_{per}(D)$ is a 
Lipschitz continuous mapping and has linear growth, that is, 
there exists a constant $C > 0$ such that for all ${\bf v}, {\bf w} \in \vL^2_{per}(D)$ 
\begin{subequations}\label{eq2.6}
	\begin{align}\label{eq2.6a}
	\|{\bf B}({\bf v})-{\bf B}({\bf w})\|  &\leq C\|{\bf v}-{\bf w}\|\, , \\
	\|{\bf B}({\bf v})\|  &\leq C \bigl( \|{\bf v}\|+1\bigr)\, ,   \label{eq2.6b}
	\end{align}
\end{subequations}
 
In this paper, we shall use $C$ to denote a generic positive constant
which may depend on $\nu, T$, the datum functions $\vu_0$, $\vf$, and the domain $D$ 
but is independent of the mesh parameter $h$ and $k$. 
In addition, unless it is stated otherwise, we assume that $\vf \in L^{q}(\Ome;C^{\frac12}(0,T; \vH^{-1}(D)))$ for some  $\forall q \in [1,\infty)$.

\subsection{Some useful facts and inequalities} In this subsection, we collect some well-known theorems 
and useful facts which will be used in  the later sections.

First of all, we recall the following Kolmogorov Criteria for a path-wise continuity of stochastic processes,
its proof can be found in  \cite[Theorem 3.3]{Prato1}.

\begin{theorem}\label{kolmogorov}
	Let $\vX(t), t \in [0,T]$, be a stochastic process with values in a separable Banach space $E$ such that, for some positive constant $C>0$, $\alpha > 0, \beta > 0$ and all $t,s \in [0,T]$,
	\begin{align}
	\mE\bigl[\|\vX(t) - \vX(s)\|^{\beta}\bigr] \leq C|t-s|^{1+\alpha}.
	\end{align}
	Then for each $T>0$, almost every $\omega$ and each $0< \gamma < \frac{\alpha}{\beta}$ there exists a constant $K = K(\omega,\gamma,T)$ such that
	\begin{align}\label{eq2.4}
	\|\vX(t,\omega) - \vX(s,\omega)\| \leq K|t-s|^{\gamma}\qquad\mbox{ for all } t,s \in [0,T].
	\end{align}
	Moreover, $\mE\bigl[|K|^\beta\bigr] < \infty$ for all $\beta >0$.
\end{theorem}

Next, we recall a useful inequality for martingale processes. This inequality is often referred to as the 
Burkholder-Davis-Gundy inequality in the literature, see \cite[Theorem 2.4]{Brze}.

\begin{lemma}
	Let $\pphi(t)\in L^2(D)$ be a random field for all $t \in [0,T]$. For any $q >0$, there exists a positive constant $C_b = C_b(T,q)>0$ such that 
	
	\begin{align}
		\mE\biggl[\max_{0\leq t \leq T}\biggl\|\int_0^t \pphi(\xi)\, dW(\xi)\biggr\|_{L^2}^q\biggr] \leq C_b\, \mE\biggl[\biggl(\int_0^T \|\pphi(\xi)\|^2_{L^2}\, d\xi\biggr)^{q/2}\biggr].
	\end{align}.
\end{lemma}

The next lemma recalls the well-known It\^o isometry and introduces a related inequality for stochastic processes. 

\begin{lemma}\label{lemma23}
	Let $\pphi(t)$ be a stochastic process on $[0,T]$. Define $\displaystyle \vX_t = \int_0^t \pphi(\xi)\, dW(\xi)$. We have 
	\begin{enumerate}[(i)]
		\item If $\pphi \in L^2(\Ome;L^2(0,T; \vL^2(D)))$, then
		\begin{align}\label{ito2}
		\mE\bigl[\|\vX_t\|^2_{\vL^2}\bigr] = \mE\biggl[\int_0^t\|\pphi(\xi)\|^2_{\vL^2}\, d\xi\biggr].
		\end{align}
		\item If $\pphi \in L^q(\Ome;L^q(0,T;\vL^2(D)))$, for $q > 2$, then
		\begin{align}\label{ito4}
		\mE\bigl[\|\vX_t\|^q_{\vL^2}\bigr] \leq C(t,q)\,\mE\biggl[\int_0^t \|\pphi(\xi)\|^q_{\vL^2}\, d\xi\biggr],
		\end{align}
		where $\displaystyle C(t,q) = \frac{C_b}{2}(q-1)(q-2) t^{\frac{q}{2}} + (q-1)C_b$.
	\end{enumerate}
\end{lemma}
\begin{proof}
	The proof of \eqref{ito2} can be found in \cite{Prato1}. Below we only give a proof for \eqref{ito4}, which is based on the It\^o formula and Burkholder-Davis-Gundy inequality. 
	
	By It\^o's formula, we have
	
	\begin{align}\label{eq28}
	\mE\bigl[\|\vX_t\|^q_{\vL^2}\bigr] &\leq q\mE\biggl[\int_{0}^t \|\vX_{\tau}\|^{q-2}_{\vL^2} \bigl(\vX_{\tau}\,, \pphi(\tau)\bigr)\, dW(\tau)\biggr] \\\nonumber
	&\qquad+ \frac12 q(q-1) \mE\biggl[\int_{0}^t \|\vX_{\tau}\|^{q-2}_{\vL^2}\|\pphi(\tau)\|^2_{\vL^2}\, d\tau\biggr].
	\end{align}
	
	The expectation of the first term on the right side of \eqref{eq28}  vanishes due to the martingale property of It\^o integrals. Therefore, we obtain
	\begin{align}\label{eq29}
	\nonumber	\mE\bigl[\|\vX_t\|^q_{\vL^2}\bigr] &\leq \frac12 q(q-1) \int_{0}^t \mE\bigl[\|\vX_{\tau}\|^{q-2}_{\vL^2}\|\pphi(\tau)\|^2_{\vL^2}\bigr]\, d\tau\\\nonumber
	&= \frac12 q(q-1)\int_{0}^t \mE\biggl[\biggl\|\int_{0}^{\tau}\pphi(\xi)\, dW(\xi)\biggr\|^{q-2}_{\vL^2}\|\pphi(\tau)\|^2_{\vL^2}\biggr]\, d\tau\\
	&\leq \frac12 p(p-1)\int_{0}^t\biggl(\mE\biggl[\biggl\|\int_{0}^{\tau}\pphi(\xi)\, dW(\xi)\biggr\|^{\alpha(q-2)}_{\vL^2}\biggr]\biggr)^{\frac{1}{\alpha}} \bigl(\mE\bigl[\|\pphi(\tau)\|^{2\beta}_{\vL^2}\bigr]\bigr)^{\frac{1}{\beta}}   \, d\tau\\\nonumber
	&\leq \frac12 q(q-1)\int_{0}^t\biggl(\mE\biggl[\max_{0\leq \tau \leq t}\biggl\|\int_{0}^{\tau}\pphi(\xi)\, dW(\xi)\biggr\|^{\alpha(q-2)}_{\vL^2}\biggr]\biggr)^{\frac{1}{\alpha}} \bigl(\mE\bigl[\|\pphi(\tau)\|^{2\beta}_{\vL^2}\bigr]\bigr)^{\frac{1}{\beta}}   \, d\tau.
	\end{align}
	We have used H\"older's inequality with $\frac{1}{\alpha} + \frac{1}{\beta} = 1$
	to obtain the second inequality.
	
	Next, applying the Burkholder-Davis-Gundy inequality to the last line of \eqref{eq29}, we get
	\begin{align*}
	\mE\bigl[\|\vX_t\|^q_{\vL^2}\bigr] &\leq \frac12 q(q-1)C_b\int_{0}^t\biggl(\mE\biggl[\biggl(\int_{0}^{t}\|\pphi(\xi)\|^2_{\vL^2}\, d\xi\biggr)^{\frac{\alpha(q-2)}{2}}\biggr]\biggr)^{\frac{1}{\alpha}} \bigl(\mE\bigl[\|\pphi(\tau)\|^{2\beta}_{\vL^2}\bigr]\bigr)^{\frac{1}{\beta}}   \, d\tau.
	\end{align*}
	
	Setting $\alpha = \frac{q}{q-2}, \beta = \frac{q}{2}$ and using Young's inequality with the conjugate 
	pair $\frac{q-2}{q}$ and $\frac{2}{q}$, we obtain 
	\begin{align*}
	\mE\bigl[\|\vX_t\|^q_{\vL^2}\bigr] &\leq \frac{C_b}{2} q(q-1)\int_{0}^t\biggl(\mE\biggl[\biggl(\int_{0}^{t}\|\pphi(\xi)\|^2_{\vL^2}\, d\xi\biggr)^{\frac{q}{2}}\biggr]\biggr)^{\frac{q-2}{q}} \bigl(\mE\bigl[\|\pphi(\tau)\|^{q}_{\vL^2}\bigr]\bigr)^{\frac{2}{q}}   \, d\tau\\\nonumber
	&\leq \frac{1}{2}q(q-1)C_b t \mE\biggl[\biggl(\int_{0}^t\|\pphi(\tau)\|^2_{\vL^2}\, d\tau\biggr)^{\frac{q}{2}}\biggr]\frac{q-2}{q}  \\\nonumber
	&\qquad\qquad\qquad+ \frac12 q(q-1)C_b \int_{0}^t \frac{\mE\bigl[\|\pphi(\tau)\|^q_{\vL^2}\bigr]}{q/2}\, d\tau\\\nonumber
	&\leq \biggl(\frac{1}{2}(q-1)(q-2)t^{\frac{q}{2}} + (q-1)\biggr)C_b \mE\biggl[\int_{0}^t \|\pphi(\tau)\|^q_{\vL^2}\, d\tau\biggr].
	\end{align*}
	The proof is complete.
\end{proof}

\smallskip

Finally, we recall the following property of the $\mathbb{R}$-valued Wiener process:
\begin{align}\label{mean_wiener}
\mE\Bigl[|W(t) - W(s)|^{2m}\Bigr] \leq C_m|t-s|^{m}\qquad\forall m \in \mathbb{N}.
\end{align}
When  $m=1$, the inequality becomes an equality with $C_m =1$. We refer the reader to \cite{Ichikawa} 
for its generalization to infinite-dimensional Wiener processes.


\subsection{Variational formulation of problem \eqref{eq1.1}} \label{sec-2.2}
We now recall the variational solution concept for \eqref{eq1.1} and refer the reader to \cite{Chow07,Prato1} for a proof of its existence and uniqueness.

\begin{definition}\label{def2.1} 
	Given $(\Omega,\cF, \{\cF_t\},\mP)$, let $W$ be an ${\mathbb R}$-valued Wiener process on it. 
	Suppose ${\bf u}_0\in L^2(\Omega, {\mathbb V})$ and $\vf \in L^2(\Ome;L^2((0,T);L^2_{per}(D)))$.
	An $\{\cF_t\}$-adapted stochastic process  $\{{\bf u}(t) ; 0\leq t\leq T\}$ is called
	a variational solution of \eqref{eq1.1} if ${\bf u} \in  L^2\bigl(\Omega; C([0,T]; {\mathbb V})) 
	\cap L^2\bigl(\Ome;0,T;\vH^2_{per}(D)\bigr)$,
	and satisfies $\mP$-a.s.~for all $t\in (0,T]$
	\begin{align}\label{eq2.8a}
	\bigl({\bf u}(t),  {\bf v} \bigr) + \nu\int_0^t  \bigl(\nab {\bf u}(s), \nab {\bf v} \bigr) 
	\,  ds&=({\bf u}_0, {\bf v})+ \int_0^t \big(\vf(s), \vv\big) \, ds \\\nonumber
	& \qquad+  {\int_0^t  \Bigl( {\bf B}\bigl({\bf u}(s)\bigr), {\bf v} \Bigr)\, dW(s)}  \qquad\forall  \, {\bf v}\in {\mathbb V}\, . 
	\end{align}

\end{definition}

Definition \ref{def2.1} only defines the velocity $\mathbf{u}$ for \eqref{eq1.1}, 
its associated pressure $p$ is subtle to define. In that regard we quote the following 
theorem from \cite{Feng1}.

\begin{theorem}\label{thm 2.2}
	Let $\{{\bf u}(t) ; 0\leq t\leq T\}$ be the variational solution of \eqref{eq1.1}. There exists a unique adapted process 
	{$P\in {L^2\bigl(\Omega, L^2(0,T; H^1_{per}(D)/{\mathbb R})\bigr)}$} such that $(\mathbf{u}, P)$ satisfies 
	$\mP$-a.s.~for all $t\in (0,T]$
	\begin{subequations}\label{eq2.100}
		\begin{align}\label{eq2.10a}
		&\bigl({\bf u}(t),  \pphi \bigr) + \nu\int_0^t  \bigl(\nab {\bf u}(s), \nab \pphi \bigr) \, ds
		- \bigl(  \div \pphi, P(t) \bigr) \\
		&=({\bf u}_0, \pphi) + \int_0^t \big(\vf(s), \pphi\big) \, ds 
		+  {\int_0^t  \bigl( {\bf B}\bigl({\bf u}(s)\bigr), \pphi \bigr)\, d\vW(s)}  \,\,\, \forall  \, \pphi\in \vH^1_{per}(D)\, , \nonumber \\ 
		&\bigl(\div {\bf u}, q \bigr) =0 \qquad\forall \, q\in L^2_0(D) := \{ q \in L^2_{per}(D):\, (q,1) = 0\}\,  .  \label{eq2.10b}
		\end{align}
	\end{subequations}
\end{theorem}

System \eqref{eq2.100} can be regarded as a mixed formulation for the stochastic Stokes system 
\eqref{eq1.1}, where the (time-averaged) pressure $P$ is defined.
Below, we also define another time-averaged ``pressure" 
\begin{align}\label{eq_R}
	R(t) := P(t) - \int_0^t\xi(s)\, dW(s),
\end{align}
  where we
use the Helmholtz decomposition ${\bf B}(\vu(t)) = \pmb{\eta}(t) + \nabla \xi(t)$, where 
$\xi\in H^1_{per}(D)/\mathbb{R}$  ${\mathbb P}\mbox{-a.s.}$ such that 
\begin{equation}\label{eq2.8f} 
\bigl(\nabla \xi(t), \nabla \phi \bigr) =  \bigl( {\bf B}(\vu(t)) , \nabla \phi 
\bigr)\qquad \forall\, \phi \in H^1_{per}(D)\, .
\end{equation}
The time averaged ``pressure" $\{R(t); 0\leq t\leq T\}$ will also be a target process to be approximated in our numerical methods in Section \ref{section_optimal}. 

The following stability estimate for the velocity $\vu$ was proved in \cite{Breit,CP12}. 

\begin{lemma}\label{stability_pde}
	 Let $\vu$ be solution defined in \eqref{eq2.8a}. Assume that $\vu_0 \in L^r\bigl(\Ome; \mV\bigr)$ for some $r \geq 2$. Then we have
	\begin{align}
	\mE\biggl[\Bigl(\sup_{0\leq t \leq T} \|\nab\vu(t)\|^2_{\vL^2} + \int_0^T \nu\|\nab^2\vu(t)\|^2_{\vL^2}\, dt\Bigr)^{\frac{r}{2}} \biggr] \leq C_r \mE\Bigl[\|\nab \vu_0\|^r_{\vL^2}\Bigr].
	\end{align}

	
\end{lemma}

Next, we introduce the H\"older continuity estimates for the variational solution $\vu$, 
a similar proof can be found in \cite{Breit, CP12} for the stochastic Navier-Stokes equations. 
we provide a proof below for completeness.

\begin{lemma}\label{lemma2.3}
	Suppose ${\bf u}_0 \in L^q\bigl(\Omega; {\mathbb V}\bigr)$ and $\vf \in L^{q}(\Ome; C^{\frac{1}{2}}(0,T;\vH^{-1}(D)))$, $\, \forall q \geq 2$. For $0 < \gamma < \frac12$, there exists a constant $C \equiv C(D_T, \vu_0)>0$, such that the variational solution to problem \eqref{eq1.1} satisfies
	for $s,t \in [0,T]$
		\begin{align}
\label{eq2.20a}		 {\mathbb E}\bigl[\| {\bf u}(t)-{\bf u}(s)\|^{q}_{\mV} \bigr] 
		\leq C|t-s|^{\gamma q}.
		\end{align}
\end{lemma}

\begin{proof}
	Following \cite{CP12,Breit}, we have that the mild solution of \eqref{eq1.1} can be represented as follow:
	\begin{align}
		\vu(t) = e^{-t\vA} \vu_0 + \int_0^t e^{(t-s)\vA}\vP_{\mH} \vB(\vu(s))\, d W(s).
	\end{align}
	For $t_2 < t_1$, write $\vu(t_1) - \vu(t_2) = {\tt I + II}$ where
	\begin{align}
		{\tt I} &= \Bigl(e^{-t_1\vA} - e^{-t_2\vA}\Bigr)\vu_0, \\\nonumber
		{\tt II}&= \int_0^{t_1} e^{(t_1-s)\vA}\vP_{\mH} \vB(\vu(s))\, d\vW(s) - \int_0^{t_2} e^{(t_2-s)\vA}\vP_{\mH} \vB(\vu(s))\, d W(s).
	\end{align}
	By the standard estimates of semigroup theory, we have
	\begin{align*}
	\|\vA^a e^{-t\vA}\| \leq C t^{-a}, \qquad \|\vA^{-b}({\bf I} - e^{-t\vA})\| \leq C t^b.
	\end{align*}
	Thus,
	 \begin{align}\label{eq2.21}
	 	\|{\tt I}\|_{\mV} &= \|e^{-t_2\vA}(e^{-(t_1-t_2)\vA} - {\bf I})\vA^{\frac12}\vu_0\|_{\vL^2}\\\nonumber
	 	&\leq C(t_1-t_2)^{\gamma} \|\nab\vu_0\|_{\vL^2}. 
	 \end{align}
	 Therefore, $\displaystyle \mE[\|{\tt I}\|^q_{\mV}] \leq C(t_1-t_2)^{\gamma q}\mE\bigl[\bigl\|\vu_0\bigr\|^q_{\mV}\bigr]$.
	 
	 Next, we can write 
	 \begin{align}
	 	{\tt II} &= \int_0^{t_2} \bigl(e^{-(t_1-s)\vA} - e^{-(t_2-s)\vA}\bigr)\vB(\vu(s))\, dW(s) \\\nonumber
	 	&\qquad+ \int_{t_2}^{t_1} e^{-(t_1 - s)\vA}\vB(\vu(s))\, dW(s)  
	 	=: {\tt II_a + II_b}.
	 \end{align}
	 
	 By the Burkholder-Davis-Gundy inequality and the fact that $\|\cdot\|_{\mV} = \|\vA^{1/2}\cdot\|_{\vL^2}$, we obtain
	 \begin{align}\label{eq2.23}
	 	\bigl(\mE\bigl[\bigl\|{\tt II}_a\bigr\|^q_{\mV}\bigr]\bigr)^{1/q} &\leq C\biggl(\int_0^{t_2} \biggl(\mE\biggl[\bigl\|\bigl(e^{-(t_1-s)\vA} - e^{-(t_2-s)\vA}\bigr)\vB(\vu(s))\bigr\|^q_{\mV}\biggr]\biggr)^{2/q}\, ds\biggr)^{1/2} \\\nonumber
	 	&\leq C \biggl(\int_{0}^{t_2} \bigl\|\vA^{(1-\varepsilon)}e^{-(t_2-s)\vA}\bigr\|^2_{\mathcal{L}(\vL^2)}\\\nonumber
	 	&\qquad\times\bigl\|\vA^{-(1-\varepsilon)}\bigl(e^{-(t_1-t_2)\vA} - {\bf I}\bigr)\bigr\|^2_{\mathcal{L}(\vL^2)}\Bigl(\mE\bigl[\|\vu(s)\|^q_{\mV}\bigr]\Bigr)^{2/q}\, ds\biggr)^{1/2}\\\nonumber
	 	&\leq C(t_1-t_2)^{1-\varepsilon}\sup_{0\leq s \leq T}\Bigl(\mE\bigl[\|\vu(s)\|^q_{\mV}\bigr]\Bigr)^{1/q}\biggl(\int_{0}^{t_2}\frac{ds}{(t_2-s)^{2(1-\varepsilon)}} \biggr)^{1/2}\\\nonumber
	 	&\leq C(t_1-t_2)^{1-\varepsilon},
	 \end{align}
	 where $\frac12<\varepsilon <1$, and $\displaystyle\Bigl(\mE\bigl[\|\vu(s)\|^q_{\mV}\bigr]\Bigr)^{1/q} <C_{q}$ by Lemma \ref{stability_pde}.
	 
	 To estimate {\tt II}$_b$, we use Lemma \ref{lemma23} (ii) and then also apply Lemma \ref{stability_pde} to obtain:
	 \begin{align}\label{eq2.24}
	 	\mE\bigl[\bigl\|{\tt II_b}\bigr\|^q_{\mV}\bigr] &\leq C_q\int_{t_2}^{t_1} \mE\Bigl[\bigl\|e^{-(t_1-s)\vA}\vB(\vu(s))\bigr\|^q_{\mV}\Bigr]\, ds\\\nonumber
	 	&\leq C_q(t_1-t_2)\sup_{0\leq s \leq T}\mE\bigl[\|\vu(s)\|^q_{\mV}\bigr].
	 \end{align}
	 
	 Finally, combining \eqref{eq2.21}, \eqref{eq2.23} and \eqref{eq2.24} we obtain
	 \begin{align}
	 	\mE\bigl[\|\vu(t_1) - \vu(t_2)\|^q_{\mV}\bigr] \leq C(t_1 -t_2)^{\gamma q},
	 \end{align}
	 where $0< \gamma < \frac12$. The proof is complete.
\end{proof}

\begin{remark} 
	  We note that due to the obstruction of nonlinearity,  the estimate obtained in \cite{Breit} requires higher regularity of $\vu_0$ and $\vB \in \mathcal{L}(\vL^2_{per}, \vH^2_{per})$ to obtain the optimal order $\gamma$. On the other hand,  the estimate of \cite{CP12} is limited to the order $\frac{\gamma}{2}$ under the same assumptions as in Lemma \ref{lemma2.3} above.   
\end{remark}

\section{Semi-discretization in time}\label{section_semi} In this section, we consider the implicit Euler-Maruyama scheme for the time discretization of \eqref{eq2.8a}. 

\subsection{Formulation of the scheme and stability estimates}
We recall the Euler-Maruyama scheme for problem \eqref{eq1.1} in the following algorithm (cf. \cite{CHP12,Feng,Feng1}). Let $I_k: = \{t_n\}_{n=1}^M$ be a uniform mesh of the interval $[0,T]$ 
with the time step-size $k = \frac{T}{M}$. Note that $t_0 = 0$ and $t_M = T$.

\smallskip
\noindent
\textbf{Algorithm 1} 

Let $\vu^0 = \vu_0$ be a given $\mV$-valued random variable. Find the pair $\{\vu^{n+1}, p^{n+1}\} \in \mV\times L^2_{per}$ recursively such that $\mP$-a.s. 
\begin{subequations}\label{eulerscheme}
	\begin{align}
	\label{eulerscheme1}	\bigl(\vu^{n+1} - \vu^n, \pphi\bigr) + \nu k \bigl(\nab\vu^{n+1}, \nab\pphi\bigr) &- k\bigl(p^{n+1}, \div \pphi\bigr) \\\nonumber &= k\bigl(\vf^{n+1},\pphi\bigr) +\bigl(\vB(\vu^n)\Delta W_{n+1},\pphi\bigr),\\
	\label{eulerscheme2}	\bigl(\div \vu^{n+1},\psi\bigr) &=0 
	\end{align}
\end{subequations}
for all $\pphi \in \vH^1_{per}(D)$ and $\psi \in L^2_{per}(D)$. Where $\displaystyle\vf^{n+1} := \vf(t_{n+1})$.

The following stability estimates for the velocity approximation $\{\vu^n\}$ of Algorithm 1 were proved in 
in \cite[Lemma 3.1]{CP12}.

\begin{lemma}\label{stability_mean}
	Let $\vu_0 \in L^{2^q}(\Ome;\mV)$ for an integer $1 \leq q <\infty$ be given, such that $\mE\bigl[\|\vu_0\|^{2^q}_{\mV}\bigr] \leq C$. Then there exists a constant $C_{T,q} = C(T, q, \vu_0)$ such that the following estimations hold:
	\begin{enumerate}[\rm (i)]
		\item\qquad $\displaystyle\mE\biggl[\max_{1\leq n \leq M}\|\vu^n\|^{2^q}_{\mV} + \nu k\sum_{n=1}^M \|\vu^n\|^{2^q-2}_{\mV}\|{\bf A}\vu^n\|^2_{\vL^2}\biggr] \leq C_{T,q}$.
		\item\qquad $\displaystyle\mE\Biggl[\biggl(\sum_{n=1}^M \|\vu^{n} - \vu^{n-1}\|^2_{\mV}\biggr)^q + \biggl(\nu k\sum_{n=1}^M \|{\bf A}\vu^{n}\|^2_{\mV}\biggr)^q\Biggr] \leq C_{T,q}$.
	\end{enumerate}
\end{lemma}


Next, we want to derive some high moment stability estimates for the pressure approximation $\{p^n\}$ of Algorithm 1, which plays a crucial role in the error analysis of the full-discrete scheme later. 

\begin{lemma}\label{stability_pressure}
	Let $\{ ({\bf u}^{m+1}, p^{m+1})\}_n$ be generated by {\rm Algorithm 1}. Assume that $\vu_0 \in L^q(\Ome;\mV)$ for $1\leq q < \infty$. Then, there exists a constant $C>0 $ such that 
	\begin{enumerate}[\rm (i)]
		\item if $\vB: \vL^2\rightarrow \mV$, then 
	\begin{align}\label{stability_pressure1}
	\mE\biggl[\biggl(k\sum_{n=1}^M \|\nab p^{n}\|^2_{\vL^2}\biggr)^q\biggr] \leq C_{T,q};
	\end{align}
		\item if $\vB: \vL^2\rightarrow \vH^1_{per}$, then 
			\begin{align}\label{stability_pressure2}
		\mE\biggl[\biggl(k\sum_{n=1}^M \|\nab p^{n}\|^2_{\vL^2}\biggr)^q\biggr] \leq \frac{C_{T,q}}{k^q}.
		\end{align}
	\end{enumerate}
	
\end{lemma}

\begin{proof}
	When $q = 1$, both \eqref{stability_pressure1}, \eqref{stability_pressure2} were already shown   \cite[Lemma 3.2]{CP12}. Thus, it remains to prove them for $q >1$. 
	
	\begin{enumerate}[\rm (i)]
		\item We first multiply the strong form of \eqref{eulerscheme1} by $\nab p^{n+1}$ and use the fact that since $\vB(\vu) \in \mV$, so  $\displaystyle \bigl(\vB(\vu^n)\Delta W_{n+1}, \nab p^{n+1}\bigr) = 0$ to conclude that
		\begin{align}\label{eq36}
		k\|\nab p^{n+1}\|^2_{\vL^2} &\leq C k\|\vf^{n+1}\|^2_{\vL^2}.
		\end{align}
		
		Next, taking the summation over the index $n$ followed by taking the $q$th power and expectation on both sides of \eqref{eq36} leads to the desired estimate. 
		  
		\item Let $\vB \in L^{\infty}(0,T;\vH^1_{per}(D))$, then $\displaystyle \bigl(\vB(\vu^n)\Delta W_{n+1}, \nab p^{n+1}\bigr) \neq 0$. Hence,
		\begin{align}\label{eq3.8}
		k\|\nab p^{n+1}\|^2_{\vL^2} \leq C k\|\vf^{n+1}\|^2_{\vL^2}+\frac{C}{k}\|\vB(\vu^n)\Delta W_{n+1}\|^2_{\vL^2}.
		\end{align}
		
		Taking the summation over the index $n$ followed by taking the $q$th power and expectation on both sides of \eqref{eq3.8}, we get 
		\begin{align}\label{eq3.9}
		\mE\biggl[\biggl(k\sum_{n=1}^M \|\nab p^{n}\|^2_{\vL^2}\biggr)^q\biggr] &\leq C_q\mE\biggl[\biggl(k\sum_{n=1}^M \|\vf^{n}\|^2_{\vL^2}\biggr)^q\biggr] 
		\\\nonumber
		&\qquad + \frac{C_q}{k^q} \mE\biggl[\biggl(\sum_{n=1}^M\|\vB(\vu^{n-1})\Delta W_{n}\|^2_{\vL^2}\biggr)^q\biggr].
		\end{align}
We now bound the last term on the right side of \eqref{eq3.9}. By the discrete H\"older inequality for summation and \eqref{eq2.6b}, we obtain
		\begin{align}\label{eq3.7}
		\mE\biggl[\biggl(\sum_{n=1}^M\|\vB(\vu^{n-1})\Delta W_{n}\|^2_{\vL^2}\biggr)^q\biggr] &\leq C_q\mE\biggl[\Bigl(\sum_{n=1}^M\|\vu^{n-1}\|^2_{\vL^2} |\Delta W_{n}|^2\Bigr)^q\biggr]\\\nonumber
		&\leq C_qM^{q-1}\mE\biggl[\sum_{n=1}^M\|\vu^{n-1}\|^{2q}_{\vL^2} |\Delta W_{n}|^{2q}\biggr].
		\end{align}
	\end{enumerate}
Using the tower property of the conditional expectation, the independence of the increments of the Wiener process and \eqref{mean_wiener}, we obtain
\begin{align}\label{eq38}
	\mE\bigl[\|\vu^{n-1}\|^{2q}_{\vL^2}|\Delta W_n|^{2q}\bigr] &\leq C_qk^q \mE\bigl[\|\vu^{n-1}\|^{2q}_{\vL^2}\bigr].
\end{align}

Substitute \eqref{eq3.7}, \eqref{eq38} into \eqref{eq3.9} we obtain
\begin{align}\label{eq3.10}
\mE\biggl[\biggl(k\sum_{n=1}^M \|\nab p^{n}\|^2_{\vL^2}\biggr)^q\biggr] &\leq C_q\mE\biggl[\biggl(k\sum_{n=1}^M \|\vf^{n}\|^2_{\vL^2}\biggr)^q\biggr] + \frac{C_q}{k^q} \mE\biggl[k\sum_{n=1}^M\|\vu^{n-1}\|^{2q}_{\vL^2}\biggr].
\end{align}

Finally, the proof is complete by using the assertion (i) of Lemma \ref{stability_mean}.
\end{proof}

\subsection{High moment and pathwise error estimates for the velocity approximation}
In this subsection, we present the first main result of this paper which establishes the optimal order 
high moment error estimates for the velocity approximation by Algorithm 1 and the sub-optimal order pathwise error estimate for the velocity approximation with the help of Theorem \ref{kolmogorov}.

\begin{theorem}\label{theorem_semi} 
	Let $\vu$ be the variational solution to \eqref{eq2.8a} and $\{\vu^{n}\}_{n=1}^M$ be generated by  Algorithm 1. Assume that $\vu_0 \in L^{2^q}(\Ome; \mV)$. Then there exists $C_1 = C_1(T,q,\vu_0,\vf)>0$ for any { integer} $1 \leq q < \infty$ and {real number} $0 < \gamma < \frac12$ such that
	\begin{align}\label{eq310}
		 \mE\bigl[\max_{1\leq n\leq M}\|\vu(t_n) - \vu^n\|^{2^q}_{\vL^2}\bigr] \leq C_1\, k^{2^{q} \gamma}.
	\end{align}
\end{theorem}

\begin{proof}
	When $q = 1$, the estimate was already proved in \cite{Feng,Feng1}. Thus, it remains to show \eqref{eq310} 
	for $q \geq 2$. We start with $q = 2$.
	
	Let $\ve^n =\vu(t_n) - \vu^n$. Integrating \eqref{eq2.10a} from $t_n$ to $t_{n+1}$ and choosing $\pphi \in \mV$, we obtain
	\begin{align}\label{eq_reform}
\bigl({\bf u}(t_{n+1}) - \vu(t_n),  \pphi \bigr) + \int_{t_n}^{t_{n+1}}  \bigl(\nab {\bf u}(s), \nab \pphi \bigr) \, ds 
&=\int_{t_n}^{t_{n+1}}\bigl(\vf(s),\pphi\bigr)\, ds \\\nonumber
& + {\int_{t_n}^{t_{n+1}}  \bigl( {\bf B}\bigl({\bf u}(s)\bigr), \pphi \bigr)\, d W(s)}.
\end{align}

Subtracting \eqref{eq_reform} from \eqref{eulerscheme1}, we obtain the following error equation for the velocity:
\begin{align}\label{eq3.4}
	\bigl(\ve^{n+1} - \ve^n, \pphi\bigr) + \nu k \bigl(\nab\ve^{n+1}, \nab\pphi\bigr) &= \nu\int_{t_n}^{t_{n+1}} \bigl(\nab(\vu(s) - \vu(t_{n+1})),\nab\pphi\bigr)\, ds\\\nonumber
	&+ \int_{t_n}^{t_{n+1}}\bigl(\vf(s) - \vf(t_{n+1}),\pphi\bigr)\, ds\\\nonumber
	&+ \int_{t_n}^{t_{n+1}} \bigl(\vB(\vu(s)) - \vB(\vu^n),\pphi\bigr)\, dW(s).
\end{align}
	
Choosing $\pphi = \ve^{n+1}$ in \eqref{eq3.4} and using the identity $2(a-b)a = a^2 - b^2 + (a-b)^2$,  then the left-hand side ({\tt LHS}) and right-hand side ({\tt RHS}) of \eqref{eq3.4} become
\begin{align}\label{eq3.5}
	{\tt LHS} = \frac{1}{2}\big[\|\ve^{n+1}\|^2_{\vL^2} - \|\ve^n\|^2_{\vL^2}\big] + \frac12 \|\ve^{n+1} - \ve^n\|^2_{\vL^2} + \nu k\|\nab\ve^{n+1}\|^2_{\vL^2}.
\end{align}
\begin{align}\label{eq3.6}
	{\tt RHS} &= \nu\int_{t_n}^{t_{n+1}} \bigl(\nab(\vu(s) - \vu(t_{n+1})),\nab\ve^{n+1}\bigr)\, ds\\\nonumber
	&\qquad+ \int_{t_n}^{t_{n+1}}\bigl(\vf(s) - \vf(t_{n+1}),\ve^{n+1}\bigr)\, ds\\\nonumber
	&\qquad+  \biggl(\int_{t_n}^{t_{n+1}}\bigl(\vB(\vu(s)) - \vB(\vu^n)\bigr)\, dW(s),\ve^{n+1} - \ve^n\biggr)\\\nonumber
	&\qquad+  \biggl(\int_{t_n}^{t_{n+1}}\bigl(\vB(\vu(s)) - \vB(\vu^n)\bigr)\, dW(s),\ve^n\biggr).
\end{align}

Next, multiplying  \eqref{eq3.5} and \eqref{eq3.6} by $\|\ve^{n+1}\|^2_{\vL^2}$ yields 	
\begin{align}\label{eq3.15}
	{\tt LHS} &= \frac14\bigl[\|\ve^{n+1}\|^4_{\vL^2} - \|\ve^n\|^4_{\vL^2}\bigr] + \frac14\bigl(\|\ve^{n+1}\|^2_{\vL^2} - \|\ve^{n}\|^2_{\vL^2}\bigr)^2\\\nonumber
	&\qquad +\frac12 \|\ve^{n+1} - \ve^n\|^2_{\vL^2}\|\ve^{n+1}\|^2_{\vL^2} + \nu k \|\nab\ve^{n+1}\|^2_{\vL^2}\|\ve^{n+1}\|^2_{\vL^2}.
\end{align}
\begin{align}\label{eq3.16}
		{\tt RHS} &= \nu\int_{t_n}^{t_{n+1}} \bigl(\nab(\vu(s) - \vu(t_{n+1})),\nab\ve^{n+1}\bigr)\, ds\|\ve^{n+1}\|^2_{\vL^2}\\\nonumber
		&\qquad+ \int_{t_n}^{t_{n+1}}\bigl(\vf(s) - \vf(t_{n+1}),\ve^{n+1}\bigr)\, ds\|\ve^{n+1}\|^2_{\vL^2}\\\nonumber
	&\qquad+  \biggl(\int_{t_n}^{t_{n+1}}\bigl(\vB(\vu(s)) - \vB(\vu^n)\bigr)\, d W(s),\ve^{n+1} - \ve^n\biggr)\|\ve^{n+1}\|^2_{\vL^2}\\\nonumber
	&\qquad+  \biggl(\int_{t_n}^{t_{n+1}}\bigl(\vB(\vu(s)) - \vB(\vu^n)\bigr)\, d W(s),\ve^n\biggr)\|\ve^{n+1}\|^2_{\vL^2}\\\nonumber
	& =: {\tt I + II + III + IV}.
\end{align}

Now, we estimate terms of {\tt I, II, III, IV} below.
\begin{align}\label{eq3.17}
	{\tt I} &\leq \nu\int_{t_n}^{t_{n+1}} \|\nab(\vu(t_{n+1}) - \vu(s))\|_{\vL^2}\|\nab\ve^{n+1}\|_{\vL^2}\|\ve^n\|_{\vL^2}\,ds\\\nonumber
	&\leq \nu\int_{t_n}^{t_{n+1}} \|\nab(\vu(t_{n+1}) - \vu(s))\|^2_{\vL^2}\|\ve^{n+1}\|^2_{\vL^2}\, ds \\\nonumber
	&\qquad\qquad+ \frac{\nu k }{4} \|\nab\ve^{n+1}\|^2_{\vL^2}\|\ve^{n+1}\|^2_{\vL^2}\\\nonumber
	&= \nu\int_{t_n}^{t_{n+1}} \|\nab(\vu(t_{n+1}) - \vu(s))\|^2_{\vL^2}\, ds\bigl(\|\ve^{n+1}\|^2_{\vL^2} - \|\ve^n\|^2_{\vL^2}\bigr) \\\nonumber
	&\qquad+\nu\int_{t_n}^{t_{n+1}} \|\nab(\vu(t_{n+1}) - \vu(s))\|^2_{\vL^2}\, ds \|\ve^n\|^2_{\vL^2} \\\nonumber
	&\qquad\qquad+ \frac{\nu k }{4} \|\nab\ve^{n+1}\|^2_{\vL^2}\|\ve^{n+1}\|^2_{\vL^2}\\\nonumber
	&\leq 8\biggl(\nu\int_{t_n}^{t_{n+1}}\|\nab(\vu(t_{n+1}) - \vu(s))\|^2_{\vL^2}\biggr)^2 + \frac{1}{32}\bigl(\|\ve^{n+1}\|^2_{\vL^2} - \|\ve^n\|^2_{\vL^2}\bigr)^2\\\nonumber
	&\qquad+ \frac{\nu^2}{4} \int_{t_n}^{t_{n+1}} \|\nab(\vu(t_{n+1}) - \vu(s))\|^4_{\vL^2}\, ds + k\|\ve^n\|^4_{\vL^2}\\\nonumber
	&\qquad\qquad+ \frac{\nu k }{4} \|\nab\ve^{n+1}\|^2_{\vL^2}\|\ve^{n+1}\|^2_{\vL^2}
\end{align}

By using \eqref{eq2.20a}, we obtain
\begin{align}
	\mE[{\tt I}] &\leq C\, k^{1+4\gamma} + \frac{1}{32}\mE\bigl[\bigl(\|\ve^{n+1}\|^2_{\vL^2} - \|\ve^n\|^2_{\vL^2}\bigr)^2\bigr] \\\nonumber
	&+ \frac{\nu k }{4} \mE\bigl[\|\nab\ve^{n+1}\|^2_{\vL^2}\|\ve^{n+1}\|^2_{\vL^2}\bigr]+ k\mE\bigl[\|\ve^n\|^4_{\vL^2}\bigr]
\end{align}

\begin{align}\label{eq3.19}
	{\tt II} &\leq \int_{t_n}^{t_{n+1}} \|\vf(s) - \vf(t_{n+1})\|_{\vH^{-1}}\|\nab\ve^{n+1}\|_{\vL^2}\|\ve^{n+1}\|^2_{\vL^2}\, ds\\\nonumber
	&\leq C\int_{t_n}^{t_{n+1}}\|\vf(s) - \vf(t_{n+1})\|^2_{\vH^{-1}}\|\ve^{n+1}\|^2_{\vL^2}\, ds + \frac{\nu k}{4}\|\nab\ve^{n+1}\|^2_{\vL^2}\|\ve^{n+1}\|^2_{\vL^2}\\\nonumber
	&\leq C\int_{t_n}^{t_{n+1}}\|\vf(s) - \vf(t_{n+1})\|^4_{\vH^{-1}}\, ds + \frac{1}{32}\bigl(\|\ve^{n+1}\|^2_{\vL^2} - \|\ve^{n}\|^2_{\vL^2}\bigr)^2 \\\nonumber
	&\qquad+C k\|\ve^n\|^4_{\vL^2} + \frac{\nu k}{4}\|\nab\ve^{n+1}\|^2_{\vL^2}\|\ve^{n+1}\|^2_{\vL^2}.
\end{align}

Since $\vf \in L^{2^q}(\Ome;C^{\frac12}(0,T; \vH^{-1}(D)))$ for $q = 2$, we have 
\begin{align}
	\mE[{\tt II}] &\leq Ck^{3} + \frac{1}{32}\mE\bigl[\bigl(\|\ve^{n+1}\|^2_{\vL^2} - \|\ve^{n}\|^2_{\vL^2}\bigr)^2\bigr] \\\nonumber
	&\qquad+C k\mE\bigl[\|\ve^n\|^4_{\vL^2}\bigr] + \frac{\nu k}{4}\mE\bigl[\|\nab\ve^{n+1}\|^2_{\vL^2}\|\ve^{n+1}\|^2_{\vL^2}\bigr].
\end{align}
\begin{align}\label{eq3.21}
	{\tt III} &= \biggl(\int_{t_n}^{t_{n+1}}\bigl(\vB(\vu(s)) - \vB(\vu^n)\bigr)\, d W(s),\ve^{n+1} - \ve^n\biggr)\|\ve^{n+1}\|^2_{\vL^2}\\\nonumber
	&\leq \Bigl\|\int_{t_n}^{t_{n+1}}\bigl(\vB(\vu(s)) -\vB(\vu^n)\bigr)\, dW(s)\Bigr\|^2_{\vL^2}\|\ve^{n+1}\|^2_{\vL^2} \\\nonumber
	&\qquad+ \frac14\|\ve^{n+1} - \ve^n\|^2_{\vL^2}\|\ve^{n+1}\|^2_{\vL^2}\\\nonumber
	&= \Bigl\|\int_{t_n}^{t_{n+1}}\bigl(\vB(\vu(s)) -\vB(\vu^n)\bigr)\, dW(s)\Bigr\|^2_{\vL^2}\bigl(\|\ve^{n+1}\|^2_{\vL^2} - \|\ve^n\|^2_{\vL^2}\bigr) \\\nonumber
	&\qquad+\Bigl\|\int_{t_n}^{t_{n+1}}\bigl(\vB(\vu(s)) -\vB(\vu^n)\bigr)\, dW(s)\Bigr\|^2_{\vL^2} \|\ve^n\|^2_{\vL^2}\\\nonumber
	&\qquad+ \frac14\|\ve^{n+1} - \ve^n\|^2_{\vL^2}\|\ve^{n+1}\|^2_{\vL^2}\\\nonumber
	&\leq 8\Bigl\|\int_{t_n}^{t_{n+1}}\bigl(\vB(\vu(s)) -\vB(\vu^n)\bigr)\, dW(s)\Bigr\|^4_{\vL^2} + \frac{1}{32}\bigl(\|\ve^{n+1}\|^2_{\vL^2} - \|\ve^{n}\|^2_{\vL^2}\bigr)^2 \\\nonumber
	&\qquad+\Bigl\|\int_{t_n}^{t_{n+1}}\bigl(\vB(\vu(s)) -\vB(\vu^n)\bigr)\, dW(s)\Bigr\|^2_{\vL^2} \|\ve^n\|^2_{\vL^2}\\\nonumber
	&\qquad+ \frac14\|\ve^{n+1} - \ve^n\|^2_{\vL^2}\|\ve^{n+1}\|^2_{\vL^2}.
\end{align}
\begin{align}\label{eq3.22}
	{\tt IV} &=\biggl(\int_{t_n}^{t_{n+1}}\bigl(\vB(\vu(s)) - \vB(\vu^n)\bigr)\, d W(s),\ve^{n}\biggr)\bigl(\|\ve^{n+1}\|^2_{\vL^2}-\|\ve^n\|^2_{\vL^2}\bigr)\\\nonumber
	&\qquad +\biggl(\int_{t_m}^{t_{m+1}}\bigl(\vB(\vu(s)) - \vB(\vu^n)\bigr)\, d W(s),\ve^n\biggr)\|\ve^{n}\|^2_{\vL^2}\\\nonumber
	&\leq 8\Bigl\|\int_{t_n}^{t_{n+1}}\bigl(\vB(\vu(s)) - \vB(\vu^n)\bigr)\, dW(s)\Bigr\|^2_{\vL^2}\|\ve^n\|^2_{\vL^2} \\\nonumber
	&\qquad+ \frac{1}{32}\bigl(\|\ve^{n+1}\|^2_{\vL^2} - \|\ve^n\|^2_{\vL^2}\bigr)^2 \\\nonumber
	&\qquad+ \biggl(\int_{t_n}^{t_{n+1}}\bigl(\vB(\vu(s)) - \vB(\vu^n)\bigr)\, d W(s),\ve^n\biggr)\|\ve^{n}\|^2_{\vL^2}.
\end{align}
We note that the last term on the right side of \eqref{eq3.22} has zero expected value because of the martingale property of the It\^o integrals.

Now, substituting the above estimates for terms {\tt I, II, III, IV} into {\tt RHS} in \eqref{eq3.16} 
and taking expectation on both {\tt LHS} and {\tt RHS}, followed by absorbing the like terms of {\tt LHS} 
in \eqref{eq3.15} into those of {\tt RHS} in \eqref{eq3.16}, we obtain 
\begin{align}\label{eq3.23}
	&\frac14\mE\bigl[\|\ve^{n+1}\|^4_{\vL^2} - \|\ve^n\|^4_{\vL^2}\bigr] +\frac14 \mE\bigl[\|\ve^{n+1} - \ve^n\|^2_{\vL^2}\|\ve^{n+1}\|^2_{\vL^2}\bigr] \\\nonumber
	&\qquad\qquad\qquad+ \frac{\nu k}{2}
	\mE\bigl[\|\nab\ve^{n+1}\|^2_{\vL^2}\|\ve^{n+1}\|^2_{\vL^2}\bigr]\\\nonumber
	&\leq C\,k^{1+4\gamma} + C k\mE\bigl[\|\ve^{n}\|^4_{\vL^2}\bigr] + Ck^3 \\\nonumber
	&\qquad+ C\mE\biggl[\Bigl\|\int_{t_n}^{t_{n+1}}\bigl(\vB(\vu(s)) - \vB(\vu^n)\bigr)\, dW(s)\Bigr\|^4_{\vL^2}\biggr] \\\nonumber
	&\qquad+ C\mE\biggl[\Bigl\|\int_{t_n}^{t_{n+1}}\bigl(\vB(\vu(s)) - \vB(\vu^n)\bigr)\, dW(s)\Bigr\|^2_{\vL^2}\|\ve^n\|^2_{\vL^2}\biggr]\\\nonumber
	&\leq Ck^{1+4\gamma} + Ck\mE\bigl[\|\ve^n\|^4_{\vL^2}\bigr] + {\tt V + VI},
\end{align}
where \begin{align*}
	{\tt V} &:= C\mE\biggl[\Bigl\|\int_{t_n}^{t_{n+1}}\bigl(\vB(\vu(s)) - \vB(\vu^n)\bigr)\, dW(s)\Bigr\|^4_{\vL^2}\biggr],\\\nonumber
	{\tt VI} &:=C\mE\biggl[\Bigl\|\int_{t_n}^{t_{n+1}}\bigl(\vB(\vu(s)) - \vB(\vu^n)\bigr)\, dW(s)\Bigr\|^2_{\vL^2}\|\ve^n\|^2_{\vL^2}\biggr].
\end{align*}
	 
	 To estimate {\tt V}, we first use \eqref{ito4} and then use \eqref{eq2.20a} to get 
	\begin{align}\label{eq3.24}
		{\tt V} &= C\mE\biggl[\Bigl\|\int_{t_n}^{t_{n+1}}\bigl(\vB(\vu(s)) - \vB(\vu^n)\bigr)\, dW(s)\Bigr\|^4_{\vL^2}\biggr]\\\nonumber
		&\leq C\mE\Bigl[\int_{t_n}^{t_{n+1}} \|\vB(\vu(s))-\vB(\vu^n)\|^4_{\vL^2}\, ds\Bigr]\\\nonumber
		&\leq C \int_{t_n}^{t_{n+1}}\mE\bigl[\|\vu(s) - \vu^n\|^4_{\vL^2}\bigr]\, ds\\\nonumber
		&\leq C \int_{t_n}^{t_{n+1}} \mE\bigl[\|\vu(s) - \vu(t_n)\|^4_{\vL^2}\bigr]\, ds + Ck\mE\bigl[\|\ve^n\|^4_{\vL^2}\bigr]\\\nonumber
		&\leq Ck^{1+4\gamma} + Ck\mE\bigl[\|\ve^n\|^4_{\vL^2}\bigr].
	\end{align}
	 To estimate {\tt VI}, we use the It\^o isometry given in \eqref{ito2} and \eqref{eq2.20a} to get 
	\begin{align}\label{eq3.25}
		{\tt VI} &=C\mE\biggl[\Bigl\|\int_{t_n}^{t_{n+1}}\bigl(\vB(\vu(s)) - \vB(\vu^n)\bigr)\, dW(s)\Bigr\|^2_{\vL^2}\|\ve^n\|^2_{\vL^2}\biggr]\\\nonumber
		&= C\mE\biggl[\Bigl\|\int_{t_n}^{t_{n+1}}\bigl(\vB(\vu(s)) - \vB(\vu^n)\bigr)\|\ve^n\|_{\vL^2}\, dW(s)\Bigr\|^2_{\vL^2}\biggr]\\\nonumber
		&= C\mE\Bigl[\int_{t_n}^{t_{n+1}}\|\vB(\vu(s)) - \vB(\vu^n)\|^2_{\vL^2}\|\ve^n\|^2_{\vL^2}\, ds\Bigr]\\\nonumber
		&\leq C\int_{t_n}^{t_{n+1}}\mE\bigl[\|\vu(s) - \vu^n\|^2_{\vL^2}\|\ve^n\|^2_{\vL^2}\bigr]\, ds\\\nonumber
		&\leq C\int_{t_n}^{t_{n+1}}\mE\bigl[\|\vu(s) - \vu(t_n)\|^2_{\vL^2}\|\ve^n\|^2_{\vL^2}\bigr]\, ds  + Ck\mE\bigl[\|\ve^n\|^4_{\vL^2}\bigr]\\\nonumber
		&\leq C\int_{t_n}^{t_{n+1}}\mE\bigl[\|\vu(s) - \vu(t_n)\|^4_{\vL^2}\bigr]\, ds  + Ck\mE\bigl[\|\ve^n\|^4_{\vL^2}\bigr]\\\nonumber
		&\leq Ck^{1+4\gamma} +  Ck\mE\bigl[\|\ve^n\|^4_{\vL^2}\bigr].
	\end{align}

	Bounding {\tt V, VI} by \eqref{eq3.24} and \eqref{eq3.25} in \eqref{eq3.23}, we obtain
	\begin{align}\label{eq3.26}
		&\frac14\mE\bigl[\|\ve^{n+1}\|^4_{\vL^2} - \|\ve^n\|^4_{\vL^2}\bigr] + \frac18\mE\bigl[\bigl(\|\ve^{n+1}\|^2_{\vL^2} - \|\ve^n\|^2_{\vL^2}\bigr)^2\bigr] \\\nonumber
		&\qquad+\frac14 \mE\bigl[\|\ve^{n+1} - \ve^n\|^2_{\vL^2}\|\ve^{n+1}\|^2_{\vL^2}\bigr] + \frac{\nu k}{2}\mE\bigl[\|\nab\ve^{n+1}\|^2_{\vL^2}\|\ve^{n+1}\|^2_{\vL^2}\bigr]\\\nonumber
		&\qquad\leq Ck^{1+4\gamma} + Ck\mE\bigl[\|\ve^n\|^4_{\vL^2}\bigr].
	\end{align}
	
	Next, lowering the index $n$ in \eqref{eq3.26} by $1$ and applying the summation $\sum_{n=1}^{\ell}$ for any $1\leq \ell \leq M$, we have
	\begin{align}\label{eq3.27}
		\mE\bigl[\|\ve^{\ell}\|^4_{\vL^2}\bigr] &+ \sum_{n=1}^{\ell}\mE\bigl[\|\ve^{n} - \ve^{n-1}\|^2_{\vL^2}\|\ve^{n}\|^2_{\vL^2}\bigr] + 2\nu k\sum_{n=1}^{\ell}\mE\bigl[\|\nab\ve^n\|^2_{\vL^2}\|\ve^n\|^2_{\vL^2}\bigr] \\\nonumber
		&\leq Ck^{4\gamma} + Ck\sum_{n=1}^{\ell} \mE\bigl[\|\ve^{n-1}\|^4_{\vL^2}\bigr] \\\nonumber
		&\leq Ck^{4\gamma} e^{Ct_{\ell}},
	\end{align}
	where we have used the discrete Gronwall inequality to get the last inequality.
	
	Taking maximum over all $1\leq \ell\leq M$ to \eqref{eq3.27}, we conclude that 
	\begin{align}\label{eq3.28}
		\max_{1\leq \ell\leq M}\mE\bigl[\|\ve^{\ell}\|^4_{\vL^2}\bigr] \leq C k^{4\gamma}.
	\end{align}
	
	Since the maximum is taken outside of $\mE[\cdot]$, hence, \eqref{eq3.28} is weaker
	than the desired estimate for $q=2$. To show the stronger estimate, we follow the technique of Lemma 3.1 proof \cite{BCP12} which uses the estimate \eqref{eq3.28} as a bridge to obtain the desired estimate.
	
	To the end, substituting \eqref{eq3.17}--\eqref{eq3.22} into {\tt RHS} in \eqref{eq3.16} and equating 
	it with {\tt LHS} in \eqref{eq3.15} (without taking expectation), we obtain
	\begin{align}\label{eq3.29}
		&\frac14\bigl[\|\ve^{n+1}\|^4_{\vL^2} - \|\ve^n\|^4_{\vL^2}\bigr] +  \frac18\bigl(\|\ve^{n+1}\|^2_{\vL^2} - \|\ve^n\|^2_{\vL^2}\bigr)^2 \\\nonumber
		&\qquad+\frac14 \|\ve^{n+1} - \ve^n\|^2_{\vL^2}\|\ve^{n+1}\|^2_{\vL^2} + \frac{\nu k}{2}\|\nab\ve^{n+1}\|^2_{\vL^2}\|\ve^{n+1}\|^2_{\vL^2} \\\nonumber
		 &\quad \leq Ck\|\ve^n\|^4_{\vL^2} +  C\int_{t_n}^{t_{n+1}} \|\nab(\vu(t_{n+1}) - \vu(s))\|^4_{\vL^2}\, ds\\\nonumber
		&\qquad +C\int_{t_n}^{t_{n+1}}\|\vf(s) - \vf(t_{n+1})\|^4_{\vH^{-1}}\, ds\\\nonumber
		&\qquad + C\Bigl\|\int_{t_n}^{t_{n+1}}\bigl(\vB(\vu(s)) -\vB(\vu^n)\bigr)\, dW(s)\Bigr\|^4_{\vL^2}\\\nonumber
		&\qquad +C\Bigl\|\int_{t_n}^{t_{n+1}}\bigl(\vB(\vu(s)) - \vB(\vu^n)\bigr)\, dW(s)\Bigr\|^2_{\vL^2}\|\ve^n\|^2_{\vL^2}\\\nonumber
		&\qquad + \biggl(\int_{t_n}^{t_{n+1}}\bigl(\vB(\vu(s)) - \vB(\vu^n)\bigr)\, d W(s),\ve^n\biggr)\|\ve^{n}\|^2_{\vL^2}.
	\end{align} 
	
Applying the summation operator $\sum_{n=1}^{\ell}$ followed by $\max_{1\leq \ell\leq M}$ and taking  expectation on both sides,  on noting that the last term on the right side of \eqref{eq3.29} would not 
vanish anymore (which is the main difference of this new process compared with the proof of \eqref{eq3.28}), 
and by using \eqref{eq3.28}, we have
	\begin{align}\label{eq3.30}
		&\mE\bigl[\max_{1\leq \ell\leq M}\|\ve^{\ell}\|^4_{\vL^2}\bigr] \\\nonumber
		&\leq C\mE\biggl[\max_{1\leq \ell\leq M}\sum_{n=1}^{\ell}\biggl(\int_{t_{n-1}}^{t_{n}}\bigl(\vB(\vu(s)) - \vB(\vu^{n-1})\bigr)\, d W(s),\ve^{n-1}\biggr)\|\ve^{n-1}\|^2_{\vL^2}\biggr]\\\nonumber
		&\quad + C\,k^{4\gamma}.
	\end{align}
To bound the first term on the right side of \eqref{eq3.30}, we appeal to Burkholder-Davis-Gundy inequality to obtain
	\begin{align}\label{eq3.31}
		&\mE\biggl[\max_{1\leq \ell\leq M}\sum_{n=1}^{\ell}\biggl(\int_{t_{n-1}}^{t_{n}}\bigl(\vB(\vu(s)) - \vB(\vu^{n-1})\bigr)\, d W(s),\ve^{n-1}\biggr)\|\ve^{n-1}\|^2_{\vL^2}\biggr]\\\nonumber
		&\qquad \leq \mE\biggl[\biggl(\sum_{n=1}^M\int_{t_{n-1}}^{t_n} \|\vB(\vu(s)) - \vB(\vu^{n-1})\|^2_{\vL^2}\|\ve^{n-1}\|^6_{\vL^2}\, ds\biggr)^{1/2}\biggr] \\\nonumber
		&\qquad\leq C\mE\biggl[\biggl(\sum_{n=1}^M\int_{t_{n-1}}^{t_n} \|\vu(s) - \vu^{n-1}\|^2_{\vL^2}\|\ve^{n-1}\|^6_{\vL^2}\, ds\biggr)^{1/2}\biggr] \\\nonumber
		&\qquad \leq C\mE\biggl[\max_{1\leq \ell\leq M}\|\ve^{\ell}\|^2_{\vL^2}\biggl(\sum_{n=1}^M\int_{t_{n-1}}^{t_n}\|\vu(s)-\vu^{n-1}\|^2_{\vL^2}\|\ve^{n-1}\|^2_{\vL^2}\, ds\biggr)^{1/2}\biggr]\\\nonumber
		&\qquad \leq \frac12 \mE\bigl[\max_{1\leq \ell\leq M}\|\ve^{\ell}\|^4_{\vL^2}\bigr] + C\mE\biggl[\sum_{n=1}^M\int_{t_{n-1}}^{t_n}\|\vu(s)-\vu^{n-1}\|^2_{\vL^2}\|\ve^{n-1}\|^2_{\vL^2}\, ds\biggr]\\\nonumber
		&\qquad \leq \frac12 \mE\bigl[\max_{1\leq \ell\leq M}\|\ve^{\ell}\|^4_{\vL^2}\bigr] + C\mE\biggl[\sum_{n=1}^M\int_{t_{n-1}}^{t_n}\|\vu(s)-\vu(t_{n-1})\|^2_{\vL^2}\|\ve^{n-1}\|^2_{\vL^2}\, ds\biggr]\\\nonumber
		&\qquad \qquad +C\mE\biggl[\sum_{n=1}^M\int_{t_{n-1}}^{t_n}\|\ve^{n-1}\|^4_{\vL^2}\, ds\biggr]
		\\\nonumber
		&\qquad \leq \frac12 \mE\bigl[\max_{1\leq \ell\leq M}\|\ve^{\ell}\|^4_{\vL^2}\bigr] 
		  +Ck\sum_{n=1}^M \mE\bigl[\|\ve^{n-1}\|^4_{\vL^2}\bigr]\\\nonumber
		&\qquad\qquad +C\sum_{n=1}^M\int_{t_{n-1}}^{t_n}\bigl(\mE\bigl[\|\vu(s)-\vu(t_{n-1})\|^4_{\vL^2}\bigr]\bigr)^{1/2}\bigl(\mE\bigl[\|\ve^{n-1}\|^4_{\vL^2}\bigr]\bigr)^{1/2}\, ds\\\nonumber
		&\qquad \leq  \frac12 \mE\bigl[\max_{1\leq \ell\leq M}\|\ve^{\ell}\|^4_{\vL^2}\bigr] + Ck^{4\gamma}.
	\end{align}
	Here, we have used  \eqref{eq2.20a} and \eqref{eq3.28} to obtain the last inequality of \eqref{eq3.31}. 

Combining \eqref{eq3.31} and \eqref{eq3.30} yields the desired estimate for the case $q = 2$.

To prove the general case $3\leq q<\infty$, for the sake of notation brevity but without loss of the generality, we let $f = 0$.
Our first task is to show the following inequality by induction for any $1\leq q < \infty$: there exists a constant $c_q > 0$ such that  holds $\mP$-a.s.
\begin{align}\label{eq332}
	\frac{1}{2^q}&\bigl[\|\ve^{n+1}\|^{2^q}_{\vL^2} - \|\ve^n\|^{2^q}_{\vL^2}\bigr] \\\nonumber
	&\leq c_q k\|\ve^n\|^{2^q}_{\vL^2} + c_q \int_{t_n}^{t_{n+1}}\|\nab(\vu(s) - \vu(t_{n+1}))\|^{2^q}_{\vL^2}\, ds \\\nonumber
	&+ c_q \sum_{j=1}^q \Bigl\|\int_{t_n}^{t_{n+1}}\bigl(\vB(\vu(s)) - \vB(\vu^n)\bigr)\, dW(s)\Bigr\|^{2^j}_{\vL^2}\|\ve^n\|^{2^q - 2^j}_{\vL^2} \\\nonumber
	&+ \Bigl(\int_{t_n}^{t_{n+1}}\bigl(\vB(\vu(s)) - \vB(\vu^n)\bigr)\, dW(s),\ve^n\Bigr)\|\ve^n\|^{2^q - 2}_{\vL^2}, 
\end{align}
which has been proved to hold for $q=2,3$. 

Suppose that \eqref{eq332} holds for any fixed integer $q (>3)$ and we want to show it also
holds for $q+1$. To the end, multiplying \eqref{eq332} by $\|\ve^{n+1}\|^{2^q}_{\vL^2}$ and use again the identity $2a(a - b) = a^2 - b^2 + (a-b)^2$ we obtain
\begin{align}\label{eq333}
	\frac{1}{2^{q+1}}&\bigl[\|\ve^{n+1}\|^{2^{q+1}}_{\vL^2} - \|\ve^n\|^{2^{q+1}}_{\vL^2}\bigr]  + \frac{1}{2^{q+1}}\bigl(\|\ve^{n+1}\|^{2^q}_{\vL^2} - \|\ve^n\|^{2^q}_{\vL^2}\bigr)^2\\\nonumber
&\leq c_q k\|\ve^n\|^{2^q}_{\vL^2}\|\ve^{n+1}\|^{2^q}_{\vL^2} + c_q \int_{t_n}^{t_{n+1}}\|\nab(\vu(s) - \vu(t_{n+1}))\|^{2^q}_{\vL^2}\, ds\|\ve^{n+1}\|^{2^q}_{\vL^2} \\\nonumber
&+ c_q \sum_{j=1}^q \Bigl\|\int_{t_n}^{t_{n+1}}\bigl(\vB(\vu(s)) - \vB(\vu^n)\bigr)\, dW(s)\Bigr\|^{2^j}_{\vL^2}\|\ve^n\|^{2^q - 2^j}_{\vL^2}\|\ve^{n+1}\|^{2^q}_{\vL^2} \\\nonumber
&+ \Bigl(\int_{t_n}^{t_{n+1}}\bigl(\vB(\vu(s)) - \vB(\vu^n)\bigr)\, dW(s),\ve^n\Bigr)\|\ve^n\|^{2^q - 2}_{\vL^2}\|\ve^{n+1}\|^{2^q}_{\vL^2}\\\nonumber
&:= {\tt I + II + III + IV}.
\end{align}
For some $\delta_1, \delta_2 >0$, we have
\begin{align}\label{eq334}
{\tt I} &= c_q k\|\ve^n\|^{2^q}_{\vL^2}\bigl(\|\ve^{n+1}\|^{2^q}_{\vL^2} - \|\ve^{n}\|^{2^q}_{\vL^2}\bigr) + c_qk\|\ve^n\|^{2^{q+1}}_{\vL^2}\\\nonumber
&\leq \frac{c^2_q k^2}{4\delta_1}\|\ve^n\|^{2^{q+1}}_{\vL^2} + \delta_1\bigl(\|\ve^{n+1}\|^{2^q}_{\vL^2} - \|\ve^n\|^{2^q}_{\vL^2}\bigr)^2 + c_qk\|\ve^n\|^{2^{q+1}}_{\vL^2}\\\nonumber
&= \Bigl(\frac{c_q^2k}{4\delta_1} + c_q\Bigr)k\|\ve^n\|^{2^{q+1}}_{\vL^2} + \delta_1\bigl(\|\ve^{n+1}\|^{2^q}_{\vL^2} - \|\ve^n\|^{2^q}_{\vL^2}\bigr)^2.
\end{align}
\begin{align}\label{eq335}
	{\tt II} &= c_q \int_{t_n}^{t_{n+1}}\|\nab(\vu(s) - \vu(t_{n+1}))\|^{2^q}_{\vL^2}\, ds\bigl(\|\ve^{n+1}\|^{2^q}_{\vL^2} - \|\ve^{n}\|^{2^q}_{\vL^2}\bigr) \\\nonumber
	&\qquad+ c_q \int_{t_n}^{t_{n+1}}\|\nab(\vu(s) - \vu(t_{n+1}))\|^{2^q}_{\vL^2}\, ds \|\ve^{n}\|^{2^q}_{\vL^2}\\\nonumber
	&\leq \frac{1}{4\delta_2}\Bigl(c_q\int_{t_n}^{t_{n+1}}\|\nab(\vu(s) - \vu(t_{n+1}))\|^{2^q}_{\vL^2}\, ds\Bigr)^2 + \delta_2\bigl(\|\ve^{n+1}\|^{2^q}_{\vL^2} - \|\ve^n\|^{2^q}_{\vL^2}\bigr)^2\\\nonumber
	&\qquad+ c_q\int_{t_n}^{t_{n+1}}\|\nab(\vu(s) - \vu(t_{n+1}))\|^{2^{q+1}}_{\vL^2}\, ds + c_q k \|\ve^n\|^{2^{q+1}}_{\vL^2}\\\nonumber
	&\leq \frac{c^2_q k}{4\delta_2} \int_{t_n}^{t_{n+1}}\|\nab(\vu(s) - \vu(t_{n+1}))\|^{2^{q+1}}_{\vL^2}\, ds + \delta_2\bigl(\|\ve^{n+1}\|^{2^q}_{\vL^2} - \|\ve^n\|^{2^q}_{\vL^2}\bigr)^2\\\nonumber
	&\qquad+ c_q\int_{t_n}^{t_{n+1}}\|\nab(\vu(s) - \vu(t_{n+1}))\|^{2^{q+1}}_{\vL^2}\, ds + c_q k \|\ve^n\|^{2^{q+1}}_{\vL^2}\\\nonumber
	&= \Bigl(\frac{c^2_q k}{4\delta_2} + c_q\Bigr) \int_{t_n}^{t_{n+1}}\|\nab(\vu(s) - \vu(t_{n+1}))\|^{2^{q+1}}_{\vL^2}\, ds \\\nonumber
	&\qquad+ \delta_2\bigl(\|\ve^{n+1}\|^{2^q}_{\vL^2} - \|\ve^n\|^{2^q}_{\vL^2}\bigr)^2 +  c_q k \|\ve^n\|^{2^{q+1}}_{\vL^2}.
\end{align}
For $\alpha_1, \cdots, \alpha_q > 0$ we have
\begin{align}\label{eq336}
	{\tt III} &= c_q \sum_{j=1}^q \Bigl\|\int_{t_n}^{t_{n+1}}\bigl(\vB(\vu(s)) - \vB(\vu^n)\bigr)\, dW(s)\Bigr\|^{2^j}_{\vL^2}\|\ve^n\|^{2^q - 2^j}_{\vL^2}\bigl(\|\ve^{n+1}\|^{2^q}_{\vL^2} - \|\ve^{n}\|^{2^q}_{\vL^2}\bigr) \\\nonumber
	&\qquad+c_q \sum_{j=1}^q \Bigl\|\int_{t_n}^{t_{n+1}}\bigl(\vB(\vu(s)) - \vB(\vu^n)\bigr)\, dW(s)\Bigr\|^{2^j}_{\vL^2}\|\ve^n\|^{2^q - 2^j}_{\vL^2}\|\ve^{n}\|^{2^q}_{\vL^2}\\\nonumber
	&\leq \sum_{j=1}^q \frac{c^2_q}{4\alpha_j} \Bigl\|\int_{t_n}^{t_{n+1}}\bigl(\vB(\vu(s)) - \vB(\vu^n)\bigr)\, dW(s)\Bigr\|^{2^{j+1}}_{\vL^2}\|\ve^n\|^{2^{q+1} - 2^{j+1}}_{\vL^2}\\\nonumber
	&\qquad+ \sum_{j=1}^q \alpha_j \bigl(\|\ve^{n+1}\|^{2^q}_{\vL^2} - \|\ve^{n}\|^{2^q}_{\vL^2}\bigr)^2\\\nonumber
	&\qquad+ c_q \sum_{j=1}^q \Bigl\|\int_{t_n}^{t_{n+1}}\bigl(\vB(\vu(s)) - \vB(\vu^n)\bigr)\, dW(s)\Bigr\|^{2^j}_{\vL^2}\|\ve^n\|^{2^{q+1} - 2^j}_{\vL^2}.
\end{align}
Similarly, for some $\delta_3 > 0$ we have
\begin{align}\label{eq337}
	{\tt IV} &= \Bigl(\int_{t_n}^{t_{n+1}}\bigl(\vB(\vu(s)) - \vB(\vu^n)\bigr)\, dW(s),\ve^n\Bigr)\|\ve^n\|^{2^q - 2}_{\vL^2}\bigl(\|\ve^{n+1}\|^{2^q}_{\vL^2} - \|\ve^{n}\|^{2^q}_{\vL^2}\bigr)\\\nonumber
	&\qquad + \Bigl(\int_{t_n}^{t_{n+1}}\bigl(\vB(\vu(s)) - \vB(\vu^n)\bigr)\, dW(s),\ve^n\Bigr)\|\ve^n\|^{2^q - 2}_{\vL^2}\|\ve^{n}\|^{2^q}_{\vL^2}\\\nonumber
	&\leq \frac{1}{4\delta_3}\Bigl\|\int_{t_n}^{t_{n+1}}\bigl(\vB(\vu(s)) - \vB(\vu^n)\bigr)\, dW(s)\Bigr\|^2_{\vL^2}\|\ve^n\|^2_{\vL^2}\|\ve^n\|^{2^{q+1} - 4}_{\vL^2} \\\nonumber
	&\qquad+ \delta_3\bigl(\|\ve^{n+1}\|^{2^q}_{\vL^2} - \|\ve^n\|^{2^q}_{\vL^2}\bigr)^2\\\nonumber
	&\qquad + \Bigl(\int_{t_n}^{t_{n+1}}\bigl(\vB(\vu(s)) - \vB(\vu^n)\bigr)\, dW(s),\ve^n\Bigr)\|\ve^n\|^{2^{q+1} - 2}_{\vL^2}\\\nonumber
	&= \frac{1}{4\delta_3}\Bigl\|\int_{t_n}^{t_{n+1}}\bigl(\vB(\vu(s)) - \vB(\vu^n)\bigr)\, dW(s)\Bigr\|^2_{\vL^2}\|\ve^n\|^{2^{q+1} - 2}_{\vL^2} \\\nonumber
	&\qquad+ \delta_3\bigl(\|\ve^{n+1}\|^{2^q}_{\vL^2} - \|\ve^n\|^{2^q}_{\vL^2}\bigr)^2\\\nonumber
	&\qquad + \Bigl(\int_{t_n}^{t_{n+1}}\bigl(\vB(\vu(s)) - \vB(\vu^n)\bigr)\, dW(s),\ve^n\Bigr)\|\ve^n\|^{2^{q+1} - 2}_{\vL^2}.
\end{align}
Substitute the estimates from \eqref{eq334}--\eqref{eq337} into \eqref{eq333} we obtain
\begin{align}\label{eq338}
	&\frac{1}{2^{q+1}}\bigl[\|\ve^{n+1}\|^{2^{q+1}}_{\vL^2} - \|\ve^n\|^{2^{q+1}}_{\vL^2}\bigr] \\\nonumber
	&\qquad\qquad+ \Bigl(\frac{1}{2^{q+1}} - \delta_1 -\delta_2 - \delta_3 - \alpha\Bigr)\bigl(\|\ve^{n+1}\|^{2^q}_{\vL^2} - \|\ve^n\|^{2^q}_{\vL^2}\bigr)^2\\\nonumber
	&\leq \Bigl(\frac{c_q^2k}{4\delta_1} + 2c_q\Bigr)k\|\ve^n\|^{2^{q+1}}_{\vL^2} + \Bigl(\frac{c^2_q k}{4\delta_2} + c_q\Bigr) \int_{t_n}^{t_{n+1}}\|\nab(\vu(s) - \vu(t_{n+1}))\|^{2^{q+1}}_{\vL^2}\, ds\\\nonumber
	&+ \sum_{j=1}^q \frac{c^2_q}{4\alpha_j} \Bigl\|\int_{t_n}^{t_{n+1}}\bigl(\vB(\vu(s)) - \vB(\vu^n)\bigr)\, dW(s)\Bigr\|^{2^{j+1}}_{\vL^2}\|\ve^n\|^{2^{q+1} - 2^{j+1}}_{\vL^2}\\\nonumber
	&+c_q \sum_{j=1}^q \Bigl\|\int_{t_n}^{t_{n+1}}\bigl(\vB(\vu(s)) - \vB(\vu^n)\bigr)\, dW(s)\Bigr\|^{2^j}_{\vL^2}\|\ve^n\|^{2^{q+1} - 2^j}_{\vL^2} \\\nonumber
	&+ \frac{1}{4\delta_3}\Bigl\|\int_{t_n}^{t_{n+1}}\bigl(\vB(\vu(s)) - \vB(\vu^n)\bigr)\, dW(s)\Bigr\|^2_{\vL^2}\|\ve^n\|^{2^{q+1} - 2}_{\vL^2} \\\nonumber
	&+\Bigl(\int_{t_n}^{t_{n+1}}\bigl(\vB(\vu(s)) - \vB(\vu^n)\bigr)\, dW(s),\ve^n\Bigr)\|\ve^n\|^{2^{q+1} - 2}_{\vL^2},
\end{align}
where $\displaystyle \alpha = \sum_{j=1}^q \alpha_j>0$. 

Now, we choose $\delta_1,\delta_2,\delta_3, \alpha>0$ such that $\frac{1}{2^{q+1}} - \delta_1 -\delta_2 - \delta_3 - \alpha >0$ so that the second term on the left side of \eqref{eq338} is positive and can be dropped at the end. Next, after rearranging terms on the right side, \eqref{eq338} infers that
\begin{align}\label{eq339}
&\frac{1}{2^{q+1}}\bigl[\|\ve^{n+1}\|^{2^{q+1}}_{\vL^2} - \|\ve^n\|^{2^{q+1}}_{\vL^2}\bigr] \\\nonumber
&\qquad\qquad+ \Bigl(\frac{1}{2^{q+1}} - \delta_1 -\delta_2 - \delta_3 - \alpha\Bigr)\bigl(\|\ve^{n+1}\|^{2^q}_{\vL^2} - \|\ve^n\|^{2^q}_{\vL^2}\bigr)^2\\\nonumber
&\leq \Bigl(\frac{c_q^2k}{4\delta_1} + 2c_q\Bigr)k\|\ve^n\|^{2^{q+1}}_{\vL^2} + \Bigl(\frac{c^2_q k}{4\delta_2} + c_q\Bigr) \int_{t_n}^{t_{n+1}}\|\nab(\vu(s) - \vu(t_{n+1}))\|^{2^{q+1}}_{\vL^2}\, ds\\\nonumber
&\qquad+ \max_{1\leq j \leq q} \frac{c^2_q}{4\alpha_j}\sum_{j=1}^q  \Bigl\|\int_{t_n}^{t_{n+1}}\bigl(\vB(\vu(s)) - \vB(\vu^n)\bigr)\, dW(s)\Bigr\|^{2^{j+1}}_{\vL^2}\|\ve^n\|^{2^{q+1} - 2^{j+1}}_{\vL^2}\\\nonumber
&\qquad+\Bigl(c_q + \frac{1}{4\delta_3}\Bigr) \sum_{j=1}^q \Bigl\|\int_{t_n}^{t_{n+1}}\bigl(\vB(\vu(s)) - \vB(\vu^n)\bigr)\, dW(s)\Bigr\|^{2^j}_{\vL^2}\|\ve^n\|^{2^{q+1} - 2^j}_{\vL^2} \\\nonumber
&\qquad+\Bigl(\int_{t_n}^{t_{n+1}}\bigl(\vB(\vu(s)) - \vB(\vu^n)\bigr)\, dW(s),\ve^n\Bigr)\|\ve^n\|^{2^{q+1} - 2}_{\vL^2}\\\nonumber
&\leq \Bigl(\frac{c_q^2k}{4\delta_1} + 2c_q\Bigr)k\|\ve^n\|^{2^{q+1}}_{\vL^2} + \Bigl(\frac{c^2_q k}{4\delta_2} + c_q\Bigr) \int_{t_n}^{t_{n+1}}\|\nab(\vu(s) - \vu(t_{n+1}))\|^{2^{q+1}}_{\vL^2}\, ds\\\nonumber
&\qquad+ \max_{1\leq j \leq q} \frac{c^2_q}{4\alpha_j}\sum_{j=1}^{q+1}  \Bigl\|\int_{t_n}^{t_{n+1}}\bigl(\vB(\vu(s)) - \vB(\vu^n)\bigr)\, dW(s)\Bigr\|^{2^{j}}_{\vL^2}\|\ve^n\|^{2^{q+1} - 2^{j}}_{\vL^2}\\\nonumber
&\qquad+\Bigl(c_q + \frac{1}{4\delta_3}\Bigr) \sum_{j=1}^{q+1} \Bigl\|\int_{t_n}^{t_{n+1}}\bigl(\vB(\vu(s)) - \vB(\vu^n)\bigr)\, dW(s)\Bigr\|^{2^j}_{\vL^2}\|\ve^n\|^{2^{q+1} - 2^j}_{\vL^2} \\\nonumber
&\qquad+\Bigl(\int_{t_n}^{t_{n+1}}\bigl(\vB(\vu(s)) - \vB(\vu^n)\bigr)\, dW(s),\ve^n\Bigr)\|\ve^n\|^{2^{q+1} - 2}_{\vL^2}\\\nonumber
&= \Bigl(\frac{c_q^2k}{4\delta_1} + 2c_q\Bigr)k\|\ve^n\|^{2^{q+1}}_{\vL^2} + \Bigl(\frac{c^2_q k}{4\delta_2} + c_q\Bigr) \int_{t_n}^{t_{n+1}}\|\nab(\vu(s) - \vu(t_{n+1}))\|^{2^{q+1}}_{\vL^2}\, ds\\\nonumber
&\qquad+\Bigl(c_q + \frac{1}{4\delta_3} + \max_{1\leq j \leq q} \frac{c^2_q}{4\alpha_j}\Bigr) \sum_{j=1}^{q+1} \Bigl\|\int_{t_n}^{t_{n+1}}\bigl(\vB(\vu(s)) - \vB(\vu^n)\bigr)\, dW(s)\Bigr\|^{2^j}_{\vL^2}\|\ve^n\|^{2^{q+1} - 2^j}_{\vL^2} \\\nonumber
&\qquad+\Bigl(\int_{t_n}^{t_{n+1}}\bigl(\vB(\vu(s)) - \vB(\vu^n)\bigr)\, dW(s),\ve^n\Bigr)\|\ve^n\|^{2^{q+1} - 2}_{\vL^2}\\\nonumber
&\leq c_{q+1}k\|\ve^n\|^{2^{q+1}}_{\vL^2} + c_{q+1} \int_{t_n}^{t_{n+1}}\|\nab(\vu(s) - \vu(t_{n+1}))\|^{2^{q+1}}_{\vL^2}\, ds\\\nonumber
&\qquad+c_{q+1} \sum_{j=1}^{q+1} \Bigl\|\int_{t_n}^{t_{n+1}}\bigl(\vB(\vu(s)) - \vB(\vu^n)\bigr)\, dW(s)\Bigr\|^{2^j}_{\vL^2}\|\ve^n\|^{2^{q+1} - 2^j}_{\vL^2} \\\nonumber
&\qquad+\Bigl(\int_{t_n}^{t_{n+1}}\bigl(\vB(\vu(s)) - \vB(\vu^n)\bigr)\, dW(s),\ve^n\Bigr)\|\ve^n\|^{2^{q+1} - 2}_{\vL^2},
\end{align}
where $$c_{q+1} = \max\Bigl\{\frac{c_q^2k}{4\delta_1} + 2c_q,\, \frac{c^2_q k}{4\delta_2} + c_q,\, c_q + \frac{1}{4\delta_3} + \max_{1\leq j \leq q} \frac{c^2_q}{4\alpha_j}\Bigr\}.$$ 
Hence, the proof of \eqref{eq332} is complete.

Next, we prove the statement of the theorem for the general case $3\leq q <\infty$, which will be carried out using the same technique as that in the proof of the 4th moment (i.e. $q=2$). Taking the expectation on \eqref{eq332} and using the martingale property of It\^o integrals and the H\"older continuity in Lemma \ref{lemma2.3}, we obtain
\begin{align}\label{eq340}
	\frac{1}{2^q}&\mE\bigl[\|\ve^{n+1}\|^{2^q}_{\vL^2} - \|\ve^n\|^{2^q}_{\vL^2}\bigr]\\\nonumber
	&\leq c_q k \mE\bigl[\|\ve^n\|^{2^q}_{\vL^2}\bigr] + c_q k^{1+2^q \gamma} \\\nonumber
	&\qquad+ c_q \sum_{j=1}^q \mE\Bigl[\Bigl\|\int_{t_n}^{t_{n+1}}\bigl(\vB(\vu(s)) - \vB(\vu^n)\bigr)\, dW(s)\Bigr\|^{2^j}_{\vL^2}\|\ve^n\|^{2^q - 2^j}_{\vL^2}\Bigr]\\\nonumber
	&\leq c_q k \mE\bigl[\|\ve^n\|^{2^q}_{\vL^2}\bigr] + c_q k^{1+2^q \gamma} \\\nonumber
	&\qquad+ c_q \sum_{j=1}^q C_j\mE\Bigl[\int_{t_n}^{t_{n+1}}\bigl\|\vB(\vu(s)) - \vB(\vu^n)\bigr\|^{2^j}_{\vL^2}\, ds\|\ve^n\|^{2^q - 2^j}_{\vL^2}\Bigr],
\end{align}
where the last inequality of \eqref{eq340} is obtained by using (ii) of Lemma \ref{lemma23}. The last term on the right side of \eqref{eq340} cab be bounded as follows
\begin{align}\label{eq341}
	&c_q \sum_{j=1}^qC_j \mE\Bigl[\int_{t_n}^{t_{n+1}}\bigl\|\vB(\vu(s)) - \vB(\vu^n)\bigr\|^{2^j}_{\vL^2}\, ds\|\ve^n\|^{2^q - 2^j}_{\vL^2}\Bigr] \\\nonumber
	&\leq c_q \sum_{j=1}^q C_j\mE\Bigl[\int_{t_n}^{t_{n+1}}\bigl\|\vB(\vu(s)) - \vB(\vu(t_n))\bigr\|^{2^j}_{\vL^2}\, ds\|\ve^n\|^{2^q - 2^j}_{\vL^2}\Bigr] \\\nonumber
	&\qquad+ c_q \sum_{j=1}^q C_j\mE\Bigl[\int_{t_n}^{t_{n+1}}\bigl\|\vB(\vu(t_n)) - \vB(\vu^n)\bigr\|^{2^j}_{\vL^2}\, ds\|\ve^n\|^{2^q - 2^j}_{\vL^2}\Bigr]\\\nonumber
	&\leq  c_q \sum_{j=1}^q C_j\mE\Bigl[\int_{t_n}^{t_{n+1}}\|\vu(s) - \vu(t_n)\|^{2^j}_{\vL^2}\, ds\|\ve^n\|^{2^q - 2^j}_{\vL^2}\Bigr] \\\nonumber
	&\qquad+ c_q \sum_{j=1}^q C_j\mE\Bigl[\int_{t_n}^{t_{n+1}}\|\ve^n\|^{2^j}_{\vL^2}\, ds\|\ve^n\|^{2^q - 2^j}_{\vL^2}\Bigr]\\\nonumber
	&=  c_q \sum_{j=1}^{q-1} C_j\mE\Bigl[\int_{t_n}^{t_{n+1}}\|\vu(s) - \vu(t_n)\|^{2^j}_{\vL^2}\, ds\|\ve^n\|^{2^q - 2^j}_{\vL^2}\Bigr] \\\nonumber
	&\qquad+ \tilde{c}_q \mE\Bigl[\int_{t_n}^{t_{n+1}} \|\vu(s) - \vu(t_n)\|^{2^q}_{\vL^2}\, ds\Bigr] + \tilde{c}_qk\mE\bigl[\|\ve^n\|^{2^q}_{\vL^2}\bigr]\\\nonumber
	&\leq c_q \sum_{j=1}^{q-1} C_j\mE\Bigl[\int_{t_n}^{t_{n+1}}\|\vu(s) - \vu(t_n)\|^{2^j}_{\vL^2}\, ds\|\ve^n\|^{2^q - 2^j}_{\vL^2}\Bigr] \\\nonumber 
	&\qquad+ \tilde{c}_q k^{1+2^q\gamma} + \tilde{c}_qk\mE\bigl[\|\ve^n\|^{2^q}_{\vL^2}\bigr].
\end{align}
In addition, for each $1 \leq j < q$, using Young's inequality with the conjugates $a = 2^{q-j}$ and $b =\frac{2^{q-j}}{2^{q-j} - 1}$ to the first term on the right side of \eqref{eq341}, we get
\begin{align}\label{eq342}
	 &c_q \sum_{j=1}^{q-1} C_j\mE\Bigl[\int_{t_n}^{t_{n+1}}\|\vu(s) - \vu(t_n)\|^{2^j}_{\vL^2}\, ds\|\ve^n\|^{2^q - 2^j}_{\vL^2}\Bigr] \\\nonumber
	 &\leq c_q\sum_{j=1}^{q-1}\frac{C_j}{a}\mE\Bigl[\int_{t_n}^{t_{n+1}}\|\vu(s) - \vu(t_n)\|^{a2^j}_{\vL^2}\, ds\Bigr] + c_q\sum_{j=1}^{q-1}\frac{C_j}{b}k\mE\bigl[\|\ve^n\|^{(2^q-2^j)b}_{\vL^2}\bigr]\\\nonumber
	 &\leq \tilde{c}_q k^{1+2^q\gamma} + \tilde{c}_q k\mE\bigl[\|\ve^q\|^{2^q}_{\vL^2}\bigr], 
\end{align}

Finally, substituting \eqref{eq342} to \eqref{eq341} and then combining it with \eqref{eq340} yield 
\begin{align}\label{eq343}
	\frac{1}{2^{q}}\mE\bigl[\|\ve^{n+1}\|^{2^q}_{\vL^2} - \|\ve^n\|^{2^q}_{\vL^2}\bigr] \leq \tilde{c}_q k\mE\bigl[\|\ve^n\|^{2^q}_{\vL^2}\bigr] + \tilde{c}_q k^{1+ 2^q\gamma}.
\end{align}
Summing \eqref{eq343} in $n$ and then using the discrete Gronwall inequality, we get
\begin{align}\label{eq344}
	\frac{1}{2^q}\mE\bigl[\|\ve^{\ell}\|^{2^q}_{\vL^2}\bigr] &\leq \tilde{c}_q k \sum_{n=1}^{\ell}\mE\bigl[\|\ve^{n-1}\|^{2^q}_{\vL^2}\bigr] + \tilde{c}_q C_{t_{\ell}} k^{2^q \gamma}\\\nonumber
	&\leq \tilde{c}_q C_{t_{\ell}} k^{2^q \gamma}\, e^{\tilde{c}_q t_{\ell}}.
\end{align}
Thus, 
\begin{align*}
\max_{1\leq\ell\leq M}
\mE\bigl[\|\ve^{\ell}\|^{2^q}_{\vL^2}\bigr] \leq Ck^{ 2^q\gamma}.
\end{align*}
Repeating the last part of the proof of the case $q=2$ we subsequently obtain
	\begin{align*}
		\mE\bigl[\max_{1\leq\ell\leq M}\|\ve^{\ell}\|^{2^q}_{\vL^2}\bigr] \leq Ck^{ 2^q\gamma}.
	\end{align*}
The proof is complete.
\end{proof}

\begin{corollary}\label{cor3.4} 
	Under the assumptions of Theorem \ref{theorem_semi}. For any real numbers $2\leq q <\infty$ and $0 < \gamma < \frac12$, there holds
	\begin{align}\label{eq3.32}
	\mE\bigl[\max_{1\leq n\leq M}\|\vu(t_n) - \vu^n\|^{q}_{\vL^2}\bigr] \leq C_1\, k^{\gamma {q}},
	\end{align}
	where $C_1 = C_1(T,q,\vu_0,\vf)$.
\end{corollary}

\begin{proof}
	The proof follows from using H\"older inequality and Theorem \ref{theorem_semi}.
\end{proof}

\begin{theorem}\label{theorem_semi_veolocityH1}
	Under the assumptions of Theorem \ref{theorem_semi}, there holds  for $2\leq q <\infty$ and $0 < \gamma < \frac12$ 
	\begin{align}\label{eq3.34}
		\mE\biggl[\biggl\|\nu k\sum_{n=1}^{M}\nab(\vu(t_n) - \vu^n)\biggr\|_{\vL^2}^q\biggr] \leq C_1 k^{\gamma q},
	\end{align}
	where $C_1 = C_1(T,q,\vu_0,\vf)$.
\end{theorem}

\begin{proof} For the sake of notational brevity, we set $\nu = 1$.
	Applying the summation operator $\sum_{n=1}^{M}$ to \eqref{eq3.4}, we obtain
	\begin{align}\label{eq3.35}
		\bigl(\ve^M, \pphi\bigr) &+ \biggl( k \sum_{n=1}^{M}\nab\ve^n,\nab\pphi\biggr) 
		=\biggl(\sum_{n=1}^{M}\int_{t_{n-1}}^{t_n}\nab\bigl(\vu(s) - \vu(t_n)\bigr)\, ds,\nab\pphi\biggr)\\\nonumber
		&\hskip 1.3in + \biggl(\sum_{n=1}^{M}\int_{t_{n-1}}^{t_n}\bigl(\vf(s) - \vf(t_n)\bigr)\, ds, \pphi\biggr)\\\nonumber
		&\hskip 1.3in +  \biggl(\sum_{n=1}^{M}\int_{t_{n-1}}^{t_n}\bigl(\vB(\vu(s)) - \vB(\vu^{n-1})\bigr)\, dW(s), \pphi\biggr).
	\end{align}
	Setting $\pphi = k \sum_{n=1}^{M}\ve^n$, and using Schwarz, Young, Poincar\'e inequalities, we obtain
	\begin{align}\label{eq3.36}
	\biggl\|k \sum_{n=1}^{M}\nab\ve^n\biggr\|^2_{\vL^2} &\leq C\|\ve^M\|^2_{\vL^2} + C\biggl\|\sum_{n=1}^{M}\int_{t_{n-1}}^{t_n}\nab\bigl(\vu(s) - \vu(t_n)\bigr)\, ds\biggr\|^2_{\vL^2}\\\nonumber
	&\qquad+ C\biggl\|\sum_{n=1}^{M}\int_{t_{n-1}}^{t_n}\bigl(\vf(s) - \vf(t_n)\bigr)\, ds\biggr\|^2_{\vH^{-1}} \\\nonumber
	&\qquad+ C\biggl\|\sum_{n=1}^{M}\int_{t_{n-1}}^{t_n}\bigl(\vB(\vu(s)) - \vB(\vu^{n-1})\bigr)\, dW(s)\biggr\|^2_{\vL^2}.
	\end{align}
	Taking the $\frac{q}{2}$-power followed by expectation on both sides of \eqref{eq3.36}, we get 
		\begin{align}\label{eq3.37}
	\mE\biggl[\biggl\|k \sum_{n=1}^{M}\nab\ve^n\biggr\|^q_{\vL^2}\biggr] &\leq C_q\mE\bigl[\|\ve^M\|^q_{\vL^2}\bigr] \\\nonumber
	&\qquad+ C_q\mE\biggl[\biggl\|\sum_{n=1}^{M}\int_{t_{n-1}}^{t_n}\nab\bigl(\vu(s) - \vu(t_n)\bigr)\, ds\biggr\|^q_{\vL^2}\biggr]\\\nonumber
	&\qquad+ C_q\mE\biggl[\biggl\|\sum_{n=1}^{M}\int_{t_{n-1}}^{t_n}\bigl(\vf(s) - \vf(t_n)\bigr)\, ds\biggr\|^q_{\vH^{-1}}\biggr] \\\nonumber
	&\qquad+ C_q\mE\biggl[\biggl\|\sum_{n=1}^{M}\int_{t_{n-1}}^{t_n}\bigl(\vB(\vu(s)) - \vB(\vu^{n-1})\bigr)\, dW(s)\biggr\|^q_{\vL^2}\biggr]\\\nonumber
	&=:{\tt I + II + III + IV}.
	\end{align}
	
	By using \eqref{eq310}, \eqref{eq2.20a}, and the assumption on $\vf$, we get
	\begin{align}
	{\tt I + II + III} \leq C_q k^{\gamma q}.
	\end{align}
To bound {\tt IV}, by \eqref{ito4} we have
	\begin{align}\label{eq3.39}
		{\tt IV} &\leq C_q\mE\biggl[\sum_{n=1}^M\int_{t_{n-1}}^{t_n}\|\vB(\vu(s)) - \vB(\vu^{n-1})\|^q_{\vL^2}\, ds\biggr]\\\nonumber
		&\leq C_q \mE\biggl[\sum_{n=1}^M\int_{t_{n-1}}^{t_n}\|\vu(s) - \vu(t_{n-1})\|^q_{\vL^2}\, ds\biggr] + C_q\mE\biggl[\sum_{n=1}^M\int_{t_{n-1}}^{t_n}\|\ve^{n-1}\|^q_{\vL^2}\, ds\biggr]\\\nonumber
		&\leq C_q k^{\gamma q}.
	\end{align}
Here, we have used \eqref{eq2.20a} and \eqref{eq310} to obtain the last inequality of \eqref{eq3.39}.   The proof is complete.
\end{proof}

\begin{remark}
	 The second-moment (i.e., $p=2$) error estimate in the $\vH^1$-norm  was obtained for the velocity approximation in \cite{Feng,Feng1}. Theorem \ref{theorem_semi_veolocityH1} proves a weak convergence of the high moments of the error in $\vH^1$-norm. 
	 The difficulty of obtaining the strong convergence of the high moments
	 of the error in $\vH^1$-norm is explained below. After setting $\pphi = \ve^{n+1}$ in \eqref{eq3.4}, using the binomial formula and summing over all $0\leq n < M$, we obtain a similar inequality as that in \eqref{eq3.37} but in strong form, namely,
\begin{align}\label{eq3.40}
\mE\bigl[\|\ve^M\|^q_{\vL^2}\bigr] &+\mE\biggl[\biggl(k \sum_{n=1}^{M}\|\nab\ve^n\|^2_{\vL^2}\biggr)^{q/2}\biggr] \\\nonumber
&\leq  
	 C_q\mE\biggl[\biggl(\sum_{n=1}^{M}\biggl\|\int_{t_{n-1}}^{t_n}\nab\bigl(\vu(s) - \vu(t_n)\bigr)\, ds\biggr\|^2_{\vL^2}\biggr)^{q/2}\biggr]\\\nonumber
	&\qquad+ C_q\mE\biggl[\biggl(\sum_{n=1}^{M}\biggl\|\int_{t_{n-1}}^{t_n}\bigl(\vf(s) - \vf(t_n)\bigr)\, ds\biggr\|^2_{\vH^{-1}}\biggr)^{q/2}\biggr] \\\nonumber
	&\qquad + C_q\mE\biggl[\biggl(\sum_{n=1}^{M}\biggl\|\int_{t_{n-1}}^{t_n}\bigl(\vB(\vu(s)) - \vB(\vu^{n-1})\bigr)\, dW(s)\biggr\|^2_{\vL^2}\biggr)^{q/2}\biggr].
	\end{align}
It is unclear how to bound the noise term on the right-hand side of \eqref{eq3.40}. 
\end{remark}

Finally, we are ready to state our first pathwise error estimate for the velocity approximation,
such an estimate has not been obtained before in the literature. 

\begin{theorem}
	 Assume that the assumptions of Theorem \ref{theorem_semi} hold. Let $2 < q < \infty$ and $0<\gamma < \frac{1}{2}$ such that $\gamma - \frac{1}{q} >0$. Then, for $0 < \gamma_1 < \gamma - \frac{1}{q}$, there exists a random variable $K_1 = K_1(\omega;C_1)$ with $\mE\bigl[|K_1|^{q}\bigr] <\infty$ such that there holds $\mP$-a.s.
	\begin{align}\label{eq3.41}
		\max_{1\leq n \leq M}\|\vu(t_n) - \vu^n\|_{\vL^2}  + \biggl\|\nu k\sum_{n=1}^M\nab\bigl(\vu(t_n) - \vu^n\bigr)\biggr\|_{\vL^2} \leq K_1 k^{\gamma_1}.
	\end{align}
\end{theorem}

\begin{proof}
	\eqref{eq3.41} is an immediate consequence of Corollary \ref{cor3.4}, Theorem \ref{theorem_semi_veolocityH1} and Kolmogorov Criteria, Theorem \ref{kolmogorov}.
\end{proof}

\subsection{High moment and pathwise error estimates for the pressure approximation}
In this subsection we derive  high moment and pathwise error estimates for the pressure approximation generated by Algorithm 1.  Once again, the pathwise error estimate is obtained by using the Kolmogorov Criteria, Theorem \ref{kolmogorov} and the high moment error estimates. 

\begin{theorem}\label{theorem_semi_pressure}
	Let $P(t)$ be the pressure process defined in Theorem \ref{thm 2.2} and $\{p^n\}^M_{n=1}$ be the pressure approximation generated by Algorithm 1. Assume that $\vu_0 \in L^{q}(\Ome; \mV)$. Then, for real number $0 < \gamma < \frac12$ and any integer $2\leq q <\infty$, there exists a positive constant $C_2 = C_2(C_1,\beta_0)$ such that for all $1\leq \ell \leq M$
	\begin{align}
		\mE\biggl[\biggl\| P(t_{\ell}) - k\sum_{n=1}^{\ell} p^n\biggr\|^q_{L^2}\biggr] \leq C_2 k^{\gamma q}.
	\end{align}
\end{theorem}

\begin{proof} The proof is based on the well-known inf-sup condition associated with the 
	Stokes problem. First, let us recall the inf-sup condition at the differential level, it
	says that there exists $\beta_0 >0$ such that
	\begin{align}\label{inf-sup}
		\sup_{\pphi \in \vH^1_{per}(D)}\frac{\bigl(w, \div \pphi\bigr)}{\|\nab\pphi\|_{\vL^2}} \geq \beta_0 \|w\|_{L^2} \qquad\forall w\in L^2_{per}(D).
	\end{align}
	
Now, integrating \eqref{eq2.10a} in $t$ from $0$ to $t_{\ell}$ for $1\leq \ell \leq M$, we obtain
\begin{align}\label{eq3.44}
&\bigl({\bf u}(t_{\ell}),  \pphi \bigr) + \nu\int_0^{t_{\ell}}  \bigl(\nab {\bf u}(s), \nab \pphi \bigr) \, ds
- \bigl(  \div \pphi, P(t_{\ell}) \bigr) \\\nonumber
&\quad =({\bf u}_0, \pphi) + \int_0^{t_{\ell}} \big(\vf(s), \pphi\big) \, ds 
+  {\int_0^{t_{\ell}}  \bigl( {\bf B}\bigl({\bf u}(s)\bigr), \pphi \bigr)\, d\vW(s)}  \quad \forall \pphi\in \vH^1_{per}(D),
\end{align}	
and applying $\sum_{n=1}^{\ell}$ to \eqref{eulerscheme1}, we get
\begin{align}\label{eq3.45}
	\bigl(\vu^{\ell}, \pphi\bigr) &+ \nu k\sum_{n=1}^{\ell} \bigl(\nab\vu^n,\nab\pphi\bigr) - k\sum_{n=1}^{\ell}\bigl(p^n,\div\pphi\bigr) \\\nonumber
	&=\bigl(\vu^0,\pphi\bigr) + k\sum_{n=1}^{\ell}\bigl(\vf^{n},\pphi\bigr) + \sum_{n=1}^{\ell}\bigl(\vB(\vu^{n-1})\Delta W_n,\pphi\bigr)\qquad\forall \pphi\in \vH^1_{per}(D).
\end{align}
	
Let $E_P^{m} := P(t_m) - k\sum_{n=1}^mp^n$ and recall that $\ve^m := \vu(t_m) - \vu^m$ from
the proof of Theorem \ref{theorem_semi}. Subtracting \eqref{eq3.44} from \eqref{eq3.45} yields 
	\begin{align}\label{eq3.46}
		\bigl(E_P^{\ell}, \div\pphi\bigr) &= \bigl(\ve^{\ell},\pphi\bigr) + \nu\Bigl(\sum_{n=1}^{\ell}\int_{t_{n-1}}^{t_n}\nab\bigl(\vu(s) - \vu^n\bigr)\, ds, \nab\pphi\Bigr) \\\nonumber
		&\qquad-\Bigl(\sum_{n=1}^{\ell}\int_{t_{n-1}}^{t_n}\bigl(\vf(s)-\vf^n\bigr)\, ds, \pphi\Bigr) \\\nonumber
		&\qquad- \Bigl(\sum_{n=1}^{\ell}\int_{t_{n-1}}^{t_n}\bigl(\vB(\vu(s)) - \vB(\vu^{n-1})\bigr)\, dW(s), \pphi\Bigr).
	\end{align}
Applying Schwarz and Poincare's inequality to the right side of \eqref{eq3.46}, we obtain
	\begin{align}\label{eq3.47}
		\frac{\bigl(E_P^{\ell}, \div\pphi\bigr)}{\|\nab\pphi\|_{\vL^2}} &\leq C\|\ve^{\ell}\|_{\vL^2} + \nu\biggl\|\sum_{n=1}^{\ell}\int_{t_{n-1}}^{t_n}\nab\bigl(\vu(s) - \vu^n\bigr)\, ds\biggr\|_{\vL^2} \\\nonumber
		&\qquad+\biggl\|\sum_{n=1}^{\ell}\int_{t_{n-1}}^{t_n}\bigl(\vf(s)-\vf^n\bigr)\, ds\biggr\|_{\vH^{-1}} \\\nonumber
		&\qquad+ C\biggl\|\sum_{n=1}^{\ell}\int_{t_{n-1}}^{t_n}\bigl(\vB(\vu(s)) - \vB(\vu^{n-1})\bigr)\, dW(s)\biggr\|_{\vL^2}.
	\end{align}
Then, it follows from applying \eqref{inf-sup} to the left-hand side of \eqref{eq3.47} that
    \begin{align}\label{eq3.48}
	\beta_0\|E_P^{\ell}\|_{L^2} &\leq C\|\ve^{\ell}\|_{\vL^2} + \nu\biggl\|\sum_{n=1}^{\ell}\int_{t_{n-1}}^{t_n}\nab\bigl(\vu(s) - \vu^n\bigr)\, ds\biggr\|_{\vL^2} \\\nonumber
	&\qquad+\biggl\|\sum_{n=1}^{\ell}\int_{t_{n-1}}^{t_n}\bigl(\vf(s)-\vf^n\bigr)\, ds\biggr\|_{\vH^{-1}} \\\nonumber
	&\qquad+ C\biggl\|\sum_{n=1}^{\ell}\int_{t_{n-1}}^{t_n}\bigl(\vB(\vu(s)) - \vB(\vu^{n-1})\bigr)\, dW(s)\biggr\|_{\vL^2}.
	\end{align}
	
Next, taking the $q$th power followed by taking expectation on both sides of \eqref{eq3.48} yields 
	\begin{align}\label{eq3.49}
	\beta^q_0\mE\bigl[\|E_P^{\ell}\|^q_{L^2}\bigr] &\leq C_q\mE\bigl[\|\ve^{\ell}\|^q_{\vL^2}\bigr] + C_q\mE\biggl[\biggl\|\sum_{n=1}^{\ell}\int_{t_{n-1}}^{t_n}\nab\bigl(\vu(s) - \vu^n\bigr)\, ds\biggr\|^q_{\vL^2}\biggr] \\\nonumber
	&\qquad+\mE\biggl[\biggl\|\sum_{n=1}^{\ell}\int_{t_{n-1}}^{t_n}\bigl(\vf(s)-\vf^n\bigr)\, ds\biggr\|^q_{\vH^{-1}}\biggr] \\\nonumber
	&\qquad+ C_q\mE\biggl[\biggl\|\sum_{n=1}^{\ell}\int_{t_{n-1}}^{t_n}\bigl(\vB(\vu(s)) - \vB(\vu^{n-1})\bigr)\, dW(s)\biggr\|^q_{\vL^2}\biggr]\\\nonumber
	&=:{\tt a+ b+ c+d}.
	\end{align}
	
	We now estimate four terms on the right-side of \eqref{eq3.49}. Using the estimates 
	of Corollary \ref{cor3.4}, \eqref{eq3.34}, \eqref{eq2.20a} 
	and the assumption $\vf \in L^q(\Ome;C^{\frac12}(0,T;\vH^{-1}(D)))$,  we obtain
	\begin{align*}
		{\tt a+ b+c} &\leq C_q\mE\bigl[\|\ve^{\ell}\|^q_{\vL^2}\bigr] + C_q\mE\biggl[\sum_{n=1}^{\ell}\int_{t_{n-1}}^{t_n}\|\nab\bigl(\vu(s) - \vu(t_n)\bigr)\|^q_{\vL^2}\, ds\biggr] \\\nonumber
		&\qquad +C_q\mE\biggl[\Bigl\|k\sum_{n=1}^{\ell}\nab\ve^{n}\, ds\Bigr\|^q_{\vL^2}\biggr] + C_q\mE\biggl[\sum_{n=1}^{\ell}\int_{t_{n-1}}^{t_n}\bigl\|\vf(s)-\vf^n\bigr\|^q_{\vH^{-1}}\, ds\biggr] \\\nonumber
		&\leq C k^{\gamma q}.
	\end{align*}
	
	Finally, to estimate term {\tt d}, using \eqref{ito4}, Corollary \ref{cor3.4} and \eqref{eq2.20a}, we get 
	\begin{align*}
		{\tt d} &\leq C_q\mE\biggl[\sum_{n=1}^{\ell}\int_{t_{n-1}}^{t_n}\|\vB(\vu(s)) - \vB(\vu^{n-1})\|^q_{\vL^2}\, ds\biggr]\\\nonumber
		&\leq C_q\mE\biggl[\sum_{n=1}^{\ell}\int_{t_{n-1}}^{t_n}\|\vu(s) - \vu(t_{n-1})\|^q_{\vL^2}\, ds\biggr] + C_q\mE\biggl[\sum_{n=1}^{\ell}\int_{t_{n-1}}^{t_n}\|\ve^{n-1}\|^q_{\vL^2}\, ds\biggr]\\\nonumber
		&\leq Ck^{\gamma q}.
	\end{align*}
The desired estimate follows from substituting the above estimates for terms {\tt a,b,c,d} into \eqref{eq3.49} and dividing the inequality by $\beta_0^q$. The proof is complete.
\end{proof}

Next, we state the pathwise error estimate for the pressure approximation. For the best of our knowledge, this is the first pathwise convergence result for  the pressure approximation.

\begin{theorem}
	Assume the assumptions of Theorem \ref{theorem_semi_pressure} hold. Let $2 < q < \infty$ and $0 < \gamma < \frac12$ such that $\gamma - \frac{1}{q} > 0$. Then, for $0 < \gamma_1 < \gamma - \frac{1}{q}$, there exists a random variable $K_1 = K_1(\omega;C_2)$ with $\mE\bigl[|K_1|^{q}\bigr] <\infty$ such that for all $1\leq \ell \leq M$, there holds $\mP$-a.s.
	\begin{align}\label{eq3.52}
	\biggl\| P(t_{\ell}) - k\sum_{n=1}^{\ell} p^n\biggr\|_{L^2} \leq K_1 k^{\gamma_1}.
	\end{align}
\end{theorem}

\begin{proof}
		The assertion follows immediately from an application of Theorem \ref{kolmogorov} based on the high moment error estimates of Theorem \ref{theorem_semi_pressure}.
\end{proof}

 
\section{Fully discrete mixed finite element discretization}\label{section_fullydiscrete} In this section, we formulate and 
analyze the spatial approximations of Algorithm 1 by using the mixed finite element method.

\subsection{Formulation of the mixed finite element method}
 Let $\mathcal{T}_h$ be a quasi-uniform mesh of the domain $D \subset \mathbb{R}^2$ with mesh size $h > 0$. We introduce the following finite element spaces:
\begin{align*}
\mH_h &= \bigl\{\vv_h \in {\bf C}(\overline{D}) \cap \vH^1_{per}(D);\,\vv_h \in [\mathcal{P}_i(K)]^2\qquad\forall K \in \mathcal{T}_h\bigr\},\\
L_h &= \bigl\{\psi_h \in C(\overline{D})\cap \in L^2_{per};\, \psi_h \in \mathcal{P}_j(K)\qquad\forall K \in \mathcal{T}_h\bigr\},
\end{align*}  
where $\mathcal{P}_i(K)$ denotes the space of all polynomials on $K$ of degree at most $i$. It is well-known that the mixed finite element space pair $\mH_h$ and $L_h$ must satisfies the Ladyzhenskaja-Babuska-Brezzi (LBB)  (or inf-sup condition) which is now quoted:  there exists $\beta_1>0$ such that
\begin{align}\label{inf-sup_discrete}
\sup_{\pphi_h \in \mH_h} \frac{\bigl(\div \pphi_h,\psi_h\bigr)}{\|\nab\pphi_h\|_{\vL^2}} \geq \beta_1\|\psi_h\|_{L^2}\qquad\forall \psi_h\in L_h,
\end{align}
where the constant $\beta_1$ is independent of $h$ (and $k$). 

\smallskip
\noindent
\textbf{Algorithm 2} 

Let $\vu_h^0$ be a given $\mH_h$-valued random variable. Find $\displaystyle \bigl(\vu_h^{n+1},p_h^{n+1}\bigr) \in \mH_h\times L_h$ such that $\mP$-a.s.
\begin{align}
\label{eq4.5}	\bigl(\vu_h^{n+1} &- \vu_h^n,\pphi_h\bigr) + \nu k \bigl(\nab \vu_h^{n+1},\nab\pphi_h\bigr) - k\bigl(p_h^{n+1},\div \pphi_h\bigr) \\\nonumber
&\qquad\qquad\qquad= k\bigl(\vf^{n+1},\pphi_h\bigr) + \bigl(\vB(\vu^n_h)\Delta W_{n+1},\pphi_h\bigr),\\
&\bigl(\div\vu_h^n,\psi_h\bigr) =0,
\end{align}
for all $\pphi_h \in \mH_h$ and $\psi_h\in L_h$.

\smallskip

Below we only consider the Taylor-Hood mixed finite element pair $\mH_h\times L_h$ (cf. \cite{Brezzi})  which takes $i=2$ and $j=1$ and is known to satisfy \eqref{inf-sup_discrete}. For the other LBB-stable mixed finite element spaces, the error analysis  is similar.

Next, we define $\mV_h\subset \mH_h$ as the following space of discretely divergent-free vector fields:
\begin{align*}
\mV_h = \Bigl\{\pphi_h \in \mH_h;\, \bigl(\div \pphi_h, q_h\bigr) =0\qquad\forall q_h \in L_h\Bigr\}.
\end{align*}

We notice that in general, $\mV_h$ is not a subspace of $\mV$.

Denote $\vQ_h: \vL^2_{per} \rightarrow \mV_h$ as the $L^2$-orthogonal projection, which satisfies 
\begin{align}
\bigl(\vv - \vQ_h\vv, \pphi_h\bigr) = 0\qquad\forall \pphi_h\in \mV_h.
\end{align}

In addition, we recall the following well-known interpolation estimates for the Taylor-Hood element:
\begin{align}
\|\vv - \vQ_h\vv\|_{\vL^2} + h\|\nab(\vv - \vQ_h\vv)\|_{\vL^2} &\leq C h^2\|\vA\vv\|_{\vL^2}\qquad\forall \vv\in \mV\cap\vH^2(D),\\
\|\vv - \vQ_h\vv\|_{\vL^2} &\leq Ch\|\nab\vv\|_{\vL^2}\qquad\forall \vv\in \mV\cap\vH^1(D).
\end{align} 

We also let $P_h: L^2_{per} \rightarrow L_h$ denote the $L^2$-orthogonal projection defined by   
\begin{align}
\bigl(\psi - P_h\psi, q_h\bigr) = 0\qquad\forall q_h \in L_h.
\end{align}
It is well-known that there holds 
\begin{align}
\|\psi - P_h\psi\|_{L^2} \leq Ch\|\nab \psi\|_{\vL^2}\qquad\forall \psi \in L^2_{per}(D)\cap H^1(D).
\end{align}

For the sake of notation brevity, in the rest of this section, we set $\vf = 0$.

We conclude this subsection by stating the following stability estimates for $\{\vu_h^n\}$ which were proved in \cite[Lemma 3.1]{BCP12}. 

\begin{lemma}\label{stability_FEM}
	Let $1 \leq q < \infty$ and $\vu_h^0 \in L^{2^q}(\Ome;\mH_h)$ satisfying  $ \mE\bigl[\|\vu_h^0\|^{2^q}_{\vL^2}\bigr] \leq C$. Then, there exists a pair $\bigl\{\vu_h^{n},p^{n}_h\bigr\}_{n=1}^M \subset L^{2^q}(\Ome; \mH_h\times L_h)$ that solves Algorithm 2 and satisfies
	\begin{enumerate}[(i)]
		\item\qquad $\displaystyle \mE\biggl[\max_{1\leq n \leq M}\|\vu_h^n\|^{2^q}_{\vL^2} + \nu k\sum_{n=1}^M \|\vu_h^n\|^{2^{q-1}}_{\vL^2}\|\nab\vu_h^n\|^2_{\vL^2}\biggr] \leq C_{T,q},$
		\item\qquad $\displaystyle \mE\biggl[\biggl(k\sum_{n=1}^M \|\nab\vu_h^n\|^2_{\vL^2}\biggr)^{2^{q-1}}\,\biggr] \leq C_{T,q},$
	\end{enumerate}
	where $C_{T,q} = C_{T,q}(D_T,q,\vu_h^0)$.
\end{lemma}

\subsection{High moment and pathwise error estimates for the fully discrete velocity approximation}
The goal of this subsection is to establish high  moment and pathwise error estimates for the fully discrete velocity approximation generated by Algorithm 2.

\begin{theorem}\label{theorem_fully} 
Let $2\leq q < \infty$ and $\vu_h^0 = \vQ_h \vu_0$.  Assume that $\vu_0 \in L^{q}(\Ome;\mV)$. Let $\{(\vu^n, p^n)\}$ and $\{\vu^n_h,p^n_h\}$ be the velocity and pressure approximations generated by Algorithm 1 and Algorithm 2, respectively. Then there holds
\begin{align}
\mE\Bigl[\max_{1\leq n \leq M}\|\vu^n - \vu_h^n\|^{q}_{\vL^2}\Bigr] &+ \mE\biggl[\Bigl(\nu k \sum_{n=1}^M\|\nab(\vu^n - \vu_h^n)\|^2_{\vL^2}\Bigr)^{q/2}\biggr]\\\nonumber
 &\leq C_3\biggl(h^{q} + h^{q}\mE\Bigl[\Bigl(k\sum_{n=1}^M\|\nab p^n\|^2_{\vL^2}\Bigr)^{q/2}\Bigr]\,\biggr),
\end{align}
where $C_3 = C_3(T,q,\vu_0,\vf)>0$ is independent of $k$ and $h$.
\end{theorem}	

\begin{proof}
	Let $\vE^n = \vu^n - \vu^n_h$ for $0 \leq n \leq M-1$. Subtracting \eqref{eulerscheme1} from  \eqref{eq4.5} we obtain the following error equation: 
	\begin{align}\label{eq410}
		\bigl(\vE^{n+1} - \vE^n,\pphi_h\bigr) &+ \nu k\bigl(\nab\vE^{n+1},\nab\pphi_h\bigr) - k \bigl(p^{n+1}-p^{n+1}_h,\div \pphi_h\bigr) \\\nonumber
		&= \bigl(\bigl(\vB(\vu^n) - \vB(\vu^n_h)\bigr)\Delta W_{n+1},\pphi_h\bigr)\qquad \forall \pphi_h \in \mH_h.
	\end{align}
 
	Since $\pphi_h = \vQ_h \vE^{n+1} = \vE^{n+1} - (\vu^{n+1} - \vQ_h\vu^{n+1}) \in \mV_h$, then $\bigl(p_h^{n+1}, \div \vQ_h\vE^{n+1}\bigr) = 0$. Thus, \eqref{eq410} becomes
	\begin{align}\label{eq411}
	\bigl(\vE^{n+1} &- \vE^n,\vQ_h\vE^{n+1}\bigr) + \nu k\|\nab\vE^{n+1}\|^2 \\\nonumber
	 &= \nu k \bigl(\nab\vE^{n+1}, \nab(\vu^{n+1} - \vQ_h\vu^{n+1})\bigr)-k \bigl(p^{n+1},\div \vQ_h\vE^{n+1}\bigr) \\\nonumber
	&\qquad+ \bigl(\bigl(\vB(\vu^n) - \vB(\vu^n_h)\bigr)\Delta W_{n+1},\vQ_h\vE^{n+1}\bigr).
	\end{align}
	
	Using the orthogonality of the $\vL^2$-projection and the binomial $2(a,a-b) = \|a\|^2 - \|b\|^2 + \|a-b\|^2$, the left side of \eqref{eq411} can written as follows:
	\begin{align}\label{eq412}
		{\tt LHS} = &\frac12\bigl[\|\vQ_h\vE^{n+1}\|^2_{\vL^2} - \|\vQ_h\vE^n\|^2_{\vL^2}\bigr] \\\nonumber
		&\qquad+ \frac12\|\vQ_h(\vE^{n+1} - \vE^n)\|^2_{\vL^2} + \nu k \|\nab\vE^{n+1}\|^2_{\vL^2} 
		=: {\tt RHS}, 
	\end{align}
and 
	\begin{align}\label{eq413}
		{\tt RHS} &= \nu k \bigl(\nab\vE^{n+1}, \nab(\vu^{n+1} - \vQ_h\vu^{n+1})\bigr) + k \bigl(p^{n+1},\div \vQ_h\vE^{n+1}\bigr) \\\nonumber
		&\qquad+ \bigl(\bigl(\vB(\vu^n) - \vB(\vu^n_h)\bigr)\Delta W_{n+1},\vQ_h\vE^{n+1}\bigr)\\\nonumber
		&=\nu k \bigl(\nab\vE^{n+1}, \nab(\vu^{n+1} - \vQ_h\vu^{n+1})\bigr)+k \bigl(p^{n+1},\div \vQ_h\vE^{n+1}\bigr) \\\nonumber
		&\qquad+ \bigl(\bigl(\vB(\vu^n) - \vB(\vu^n_h)\bigr)\Delta W_{n+1},\vQ_h\vE^{n+1} - \vQ_h\vE^n\bigr) \\\nonumber
		&\qquad+  \bigl(\bigl(\vB(\vu^n) - \vB(\vu^n_h)\bigr)\Delta W_{n+1},\vQ_h\vE^n\bigr)\\\nonumber
		&= {\tt I + II + III + IV}.
	\end{align} 
	
Using Cauchy-Schwarz's inequality and Young's inequality, we obtain
	\begin{align}\label{eq414}
		{\tt I} &\leq \frac{\nu k}{4}\|\nab\vE^{n+1}\|^2_{\vL^2} + \nu k \|\nab(\vu^{n+1} - \vQ_h\vu^{n+1})\|^2_{\vL^2},\\\nonumber
		&\leq \frac{\nu k}{4}\|\nab\vE^{n+1}\|^2_{\vL^2} + Ck h^2\|\vA \vu^{n+1}\|_{\vL^2}
	\end{align}
	where the first term on the right side of \eqref{eq414} will be absorbed to the left side of \eqref{eq411} later. In addition,
	\begin{align}\label{eq415}
		{\tt III} \leq \|(\vB(\vu^n) - \vB(\vu^n_h))\Delta W_{n+1}\|^2_{\vL^2} + \frac14\|\vQ_h(\vE^{n+1} - \vE^n)\|^2_{\vL^2}.
	\end{align}
	Moreover, using the fact that $\bigl(P_h p^{n+1},\div\vQ_h\vE^{n+1}\bigr) = 0$, we have
	\begin{align}\label{eq416}
	{\tt II} &= k\bigl(p^{n+1},\div \vQ_h\vE^{n+1}\bigr)\\\nonumber
	& = k\bigl(p^{n+1} - P_hp^{n+1},\div \vQ_h\vE^{n+1}\bigr)\\\nonumber
	&\leq \frac{\nu k}{4}\|\nab\vE^{n+1}\|^2_{\vL^2} + Ck\|p^{n+1} - P_h p^{n+1}\|^2_{L^2}\\\nonumber
	&\leq \frac{\nu k}{4}\|\nab\vE^{n+1}\|^2_{\vL^2} + Ckh^2\|\nab p^{n+1}\|^2_{\vL^2}.
	\end{align}

Substituting \eqref{eq412}--\eqref{eq416} into \eqref{eq411} yields 
	\begin{align}\label{eq417}
		&\frac12\bigl[\|\vQ_h\vE^{n+1}\|^2_{\vL^2} - \|\vQ_h\vE^n\|^2_{\vL^2}\bigr] + \frac14\|\vQ_h(\vE^{n+1} - \vE^n)\|^2_{\vL^2} + \frac{\nu k}{2} \|\nab\vE^{n+1}\|^2_{\vL^2}\\\nonumber
		&\leq Ckh^2\|\vA\vu^{n+1}\|^2_{\vL^2} + Ckh^2\|\nab p^{n+1}\|^2_{\vL^2} + \|(\vB(\vu^n) - \vB(\vu^n_h))\Delta W_{n+1}\|^2_{\vL^2} \\\nonumber
		&\qquad+ \bigl(\bigl(\vB(\vu^n) - \vB(\vu^n_h)\bigr)\Delta W_{n+1},\vQ_h\vE^n\bigr).
	\end{align}
Lowering one index of \eqref{eq417} and applying the summation operator $\sum_{n=1}^{\ell}$ for $1 \leq \ell \leq M$, we get
	\begin{align}\label{eq418}
		\|\vQ_h\vE^{\ell}\|^2_{\vL^2} &+ \nu k\sum_{n=1}^{\ell}\|\nab\vE^{n}\|^2_{\vL^2} \leq Ch^2 k\sum_{n=1}^{\ell}\|\vA\vu^{n}\|^2_{\vL^2} + Ch^2 k\sum_{n=1}^{\ell}\|\nab p^n\|^2_{\vL^2}\\\nonumber
		&\hskip 1.0in +2\sum_{n=1}^{\ell} \|(\vB(\vu^{n-1}) - \vB(\vu^{n-1}_h))\Delta W_{n}\|^2_{\vL^2}\\\nonumber
		&\hskip 1.0in + 2\biggl|\sum_{n=1}^{\ell}\bigl(\bigl(\vB(\vu^{n-1}) - \vB(\vu^{n-1}_h)\bigr)\Delta W_{n},\vQ_h\vE^{n-1}\bigr)\biggr|.
	\end{align}
	
	Next, taking maximum over all $1\leq \ell \leq M$ and followed by taking the $\frac{q}{2}$-power for any $2 \leq q < \infty$ and the expectation to \eqref{eq418}, we obtain
	\begin{align}\label{eq419}
	&\mE\bigl[\max_{1\leq \ell\leq M}\|\vQ_h\vE^{\ell}\|^{q}_{\vL^2}\bigr] + \mE\biggl[\Bigl(\nu k \sum_{n=1}^M\|\nab\vE^n\|^2_{\vL^2}\Bigr)^{q/2}\biggr]\\\nonumber
	 &\leq C_qh^{q} \mE\biggl[\Bigl(k\sum_{n=1}^{M}\|\vA\vu^{n}\|^2_{\vL^2}\Bigr)^{q/2}\biggr] + C_qh^{q} \mE\biggl[\Bigl(k\sum_{n=1}^{M}\|\nab p^n\|^2_{\vL^2}\Bigr)^{q/2}\biggr]\\\nonumber
	&\qquad+C_q\mE\biggl[\Bigl(\sum_{n=1}^{M} \|(\vB(\vu^{n-1}) - \vB(\vu^{n-1}_h))\Delta W_{n}\|^2_{\vL^2}\Bigr)^{q/2}\biggr]\\\nonumber
	&\qquad+ C_q\mE\biggl[\max_{1\leq \ell\leq M}\biggl|\sum_{n=1}^{\ell}\bigl(\bigl(\vB(\vu^{n-1}) - \vB(\vu^{n-1}_h)\bigr)\Delta W_{n},\vQ_h\vE^{n-1}\bigr)\biggr|^{q/2}\biggr].
	\end{align}
	
	We can use stability estimate (ii) in Lemma \ref{stability_mean} to control the first term on the right-hand side of \eqref{eq419} and Lemma \ref{stability_pressure} to bound the second term. Hence,  it remains to bound the last two terms on the right side of \eqref{eq419}. 
	Proceeding similarly as in \eqref{eq3.7} and \eqref{eq38}, we obtain
	\begin{align}\label{eq4.20}
		&\mE\biggl[\Bigl(\sum_{n=1}^{M} \|(\vB(\vu^{n-1}) - \vB(\vu^{n-1}_h))\Delta W_{n}\|^2_{\vL^2}\Bigr)^{q/2}\biggr]\\\nonumber
		 &\qquad = \mE\biggl[\Bigl(\sum_{n=1}^{M} \|\vB(\vu^{n-1}) - \vB(\vu^{n-1}_h)\|^2_{\vL^2}|\Delta W_{n}|^2\Bigr)^{q/2}\biggr]\\\nonumber
		 &\qquad \leq C_q\mE\biggl[\Bigl(\sum_{n=1}^{M} \|\vE^{n-1}\|^2_{\vL^2}|\Delta W_{n}|^2\Bigr)^{q/2}\biggr]\\\nonumber		
		 &\qquad \leq C_q M^{q/2-1}\mE\biggl[\sum_{n=1}^M \|\vE^{n-1}\|^{q}_{\vL^2}|\Delta W_{n+1}|^{q}\biggr]\\\nonumber
		 &\qquad\leq C_q M^{q/2-1} k^{q/2} \sum_{n=1}^M\mE\bigl[\|\vE^{n-1}\|^{q}_{\vL^2}\bigr]\\\nonumber
		 &\qquad\leq C_q k\sum_{n=1}^M \mE\bigl[\|\vQ_h\vE^{n-1}\|^{q}_{\vL^2}\bigr] + C_q k\sum_{n=1}^M\mE\bigl[\|\vu^{n-1} - \vQ_h\vu^{n-1}\|^{q}_{\vL^2}\bigr]\\\nonumber
		 &\qquad\leq C_q k\sum_{n=1}^M \mE\bigl[\|\vQ_h\vE^{n-1}\|^{q}_{\vL^2}\bigr] + C_q h^{q}k\sum_{n=1}^M\mE\bigl[\|\nab\vu^{n-1}\|^{q}_{\vL^2}\bigr].
	\end{align}
	
	To bound the second term on the right-hand side of \eqref{eq419}, we use the Burkholder-Davis-Gundy  and H\"older inequalities to obtain
	\begin{align}\label{eq4.21}
		&\mE\biggl[\max_{1\leq \ell\leq M}\biggl|\sum_{n=1}^{\ell}\bigl(\bigl(\vB(\vu^{n-1}) - \vB(\vu^{n-1}_h)\bigr)\Delta W_{n},\vQ_h\vE^{n-1}\bigr)\biggr|^{q/2}\biggr]\\\nonumber
		&\leq C_q\mE\biggl[\Bigl(\sum_{n=1}^M\|\vB(\vu^{n-1}) - \vB(\vu^{n-1}_h)\|^2_{\vL^2}|\Delta W_n|^2 \|\vQ_h\vE^{n-1}\|^2_{\vL^2}\Bigr)^{q/4}\biggr]\\\nonumber
		&\leq C_q \mE\biggl[\Bigl(\sum_{n=1}^M\|\vE^{n-1}\|^2_{\vL^2}\|\vQ_h\vE^{n-1}\|^2_{\vL^2}|\Delta W_n|^2\Bigr)^{q/4}\biggr]\\\nonumber
		&\leq C_q M^{q/4-1} k^{q/4}\sum_{n=1}^M \mE\bigl[\|\vE^{n-1}\|^{q/2}_{\vL^2}\|\vQ_h\vE^{n-1}\|^{q/2}_{\vL^2}\bigr]\\\nonumber
		&\leq C_qk\sum_{n=1}^M \mE\bigl[\|\vQ_h\vE^{n-1}\|^{q}_{\vL^2}\bigr] \\\nonumber
		&\qquad+ C_q k \sum_{n=1}^M \mE\bigl[\|\vu^{n-1} - \vQ_h\vu^{n-1}\|^{q/2}_{\vL^2}\|\vQ_h\vE^{n-1}\|^{q/2}_{\vL^2}\bigr] \\\nonumber
		&\leq C_qk\sum_{n=1}^M \mE\bigl[\|\vQ_h\vE^{n-1}\|^{q}_{\vL^2}\bigr] + C_q k \sum_{n=1}^M \mE\bigl[\|\vu^{n-1} - \vQ_h\vu^{n-1}\|^{q}_{\vL^2}\bigr] \\\nonumber
			&\leq C_qk\sum_{n=1}^M \mE\bigl[\|\vQ_h\vE^{n-1}\|^{q}_{\vL^2}\bigr] + C_q h^{q} k \sum_{n=1}^M \mE\bigl[\|\nab\vu^{n-1}\|^{q}_{\vL^2}\bigr].
	\end{align}
	
	Substituting \eqref{eq4.20} and \eqref{eq4.21} to the right-hand side of \eqref{eq419} yields
	\begin{align}\label{eq4.22}
	\mE\bigl[\max_{1\leq \ell\leq M}&\|\vQ_h\vE^{\ell}\|^{q}_{\vL^2}\bigr] + \mE\biggl[\Bigl(\nu k \sum_{n=1}^M\|\nab\vE^n\|^2_{\vL^2}\Bigr)^{q/2}\biggr]\\\nonumber
	 &\leq C_qh^{q} \mE\biggl[\Bigl(k\sum_{n=1}^{M}\|\vA\vu^{n}\|^2_{\vL^2}\Bigr)^{q/2}\biggr] + C_qh^{q} \mE\biggl[\Bigl(k\sum_{n=1}^{M}\|\nab p^n\|^2_{\vL^2}\Bigr)^{q/2}\biggr]\\\nonumber
	&\qquad+C_q h^{q} k\sum_{n=1}^M \mE\bigl[\|\nab\vu^{n-1}\|^{q}_{\vL^2}\bigr] + C_q k\sum_{n=1}^M\mE\bigl[\|\vQ_h\vE^{n-1}\|^{q}_{\vL^2}\bigr]\\\nonumber
	&\leq C_q h^{q} + C_q h^{q} \mE\biggl[\Bigl(k\sum_{n=1}^{M}\|\nab p^n\|^2_{\vL^2}\Bigr)^{q/2}\biggr]\\\nonumber
	&\qquad+  C_q k\sum_{n=1}^{M-1}\mE\bigl[\max_{1\leq \ell\leq n}\|\vQ_h\vE^{\ell}\|^{q}_{\vL^2}\bigr]\\\nonumber
	&\leq \biggl(C_q h^{q} + C_q h^{q} \mE\biggl[\Bigl(k\sum_{n=1}^{M}\|\nab p^n\|^2_{\vL^2}\Bigr)^{q/2}\biggr]\biggr) e^{C_qT},
	\end{align}
	where the discrete Gronwall inequality is used to obtain the last inequality. 
	
	The proof is now completed by using the triangular inequality $\|\vE^{n}\|_{\vL^2} \leq \|\vQ_h\vE^n\|_{\vL^2} + \|\vu^n - \vQ_h\vu^n\|_{\vL^2}$.
\end{proof}

We conclude this subsection by stating a pathwise error estimate for the velocity approximation by Algorithm 2 which is a direct corollary of the Kolmogorov Criteria (cf. Theorem \ref{kolmogorov}) and 
the high moment error estimates of Theorem \ref{theorem_fully}.

\begin{theorem}\label{theorem_pathwise_fully}
	Assume that the assumptions of Theorem \ref{theorem_fully} hold. Let $2\leq q <\infty$ and $0 < \gamma_2 < 1 - \frac{1}{q}$. Then, there exists a random variable $K_2 = K_2(\omega;C_3)$ with $\mE\bigl[|K_2|^q\bigr] < \infty$ such that there holds $\mP$-a.s.
	\begin{align}
	\max_{1\leq n \leq M}\|\vu^n - \vu^n_h\|_{\vL^2} + \Bigl(\nu k \sum_{n=1}^M\|\nab(\vu^n - \vu_h^n)\|^2_{\vL^2}\Bigr)^{1/2}  \leq K_2\biggl( h^{\gamma_2} + \Bigl(\frac{h}{\sqrt{k}}\Bigr)^{\gamma_2}\biggr),
	\end{align}
\end{theorem}

\subsection{High moment and pathwise error estimates for the fully discrete pressure approximation}
  In this subsection, we establish high moment and pathwise error estimates for the pressure approximation generated by Algorithm 2.  
  
\begin{theorem}\label{theorem_fully_pressure} 
	Let $2\leq q <\infty$, under the assumptions of Theorem \ref{theorem_fully}, there holds
	\begin{align}\label{eq4.24}
		\mE\biggl[\biggl\|k\sum_{n=1}^M\big(p^n - p^n_h\big)\biggr\|^{q}_{L^2}\biggr] \leq C_4\biggl(h^{q} + h^{q}\mE\Bigl[\Bigl(k\sum_{n=1}^M\|\nab p^n\|^2_{\vL^2}\Bigr)^{q/2}\Bigr]\,\biggr),
	\end{align}
	where $C_4 = C_4(\beta_1, C_3)$ and independent of $k, h$.
\end{theorem}

\begin{proof}
	The proof of \eqref{eq4.24} mimics that of Theorem \ref{theorem_semi_pressure}. 
	Let $\E_p^n = p^n - p^n_h$ and $\vE^n = \vu^n - \vu^n_h$ be the same as in Theorem \ref{theorem_fully}. Applying the summation operator $\sum_{n=1}^{M}$ to the pressure error 
	equation, we obtain
	\begin{align}\label{eq4.25}
		\Bigl(k\sum_{n=1}^M\E_p^n, \div\pphi_h\Bigr) &= \bigl(\vE^M - \vE^0,\pphi_h\bigr) + \nu k \sum_{n=1}^M\bigl(\nab\vE^n, \nab\pphi_h\bigr) \\\nonumber
		&\qquad- \sum_{n=1}^M\bigl((\vB(\vu^{n-1}) - \vB(\vu^{n-1}_h))\Delta W_n, \pphi_h\bigr).
	\end{align}
Using Schwarz inequality, we get
	\begin{align}\label{eq4.26}
		\frac{\Bigl(k\sum_{n=1}^M\E_p^n, \div\pphi_h\Bigr)}{\|\nab\pphi_h\|_{\vL^2}} &\leq C\bigl(\|\vE^M\|_{\vL^2} + \|\vE^0\|_{\vL^2}\bigr) + \nu k\sum_{n=1}^M\|\nab\vE^n\|_{\vL^2} \\\nonumber
		&\qquad+ C\sum_{n=1}^M\|\bigl(\vB(\vu^{n-1}) - \vB(\vu^{n-1}_h)\bigr)\Delta W_n\|_{\vL^2}.
	\end{align}
	
	Next, applying the discrete inf-sup condition \eqref{inf-sup_discrete} on the left-hand side yields
	\begin{align}\label{eq4.27}
	\beta_1\Bigl\|k\sum_{n=1}^M\E_p^n\Bigr\|_{L^2} &\leq C\bigl(\|\vE^M\|_{\vL^2} + \|\vE^0\|_{\vL^2}\bigr) + \nu k\sum_{n=1}^M\|\nab\vE^n\|_{\vL^2} \\\nonumber
	&\qquad+ C\sum_{n=1}^M\|\bigl(\vB(\vu^{n-1}) - \vB(\vu^{n-1}_h)\bigr)\Delta W_n\|_{\vL^2}.
	\end{align}
	Taking the $q$-power to \eqref{eq4.27} followed by taking the expectations, we obtain
	\begin{align}\label{eq4.28}
	\beta_1^{q}\mE\biggl[\Bigl\|k\sum_{n=1}^M\E_p^n\Bigr\|^{q}_{L^2}\biggr] &\leq C_q\mE\Bigl[\|\vE^M\|^{q}_{\vL^2} + \|\vE^0\|^{q}_{\vL^2}\Bigr] \\ \nonumber
	&\quad  + C_q\mE\biggl[\biggl(\nu k\sum_{n=1}^M\|\nab\vE^n\|^2_{\vL^2}\biggr)^{q/2}\biggr] \\\nonumber
	&\quad + C_q\mE\biggl[\biggl(\sum_{n=1}^M\|\bigl(\vB(\vu^{n-1}) - \vB(\vu^{n-1}_h)\bigr)\Delta W_n\|_{\vL^2}\biggr)^{q}\biggr].
	\end{align}
	
	The first three terms on the right side of \eqref{eq4.28} can be controlled by Theorem \ref{theorem_fully}, and the noise term can be bounded similarly as in \eqref{eq3.7} and \eqref{eq38} and using Theorem \ref{theorem_fully}. In summary, we obtain 
	\begin{align}
		\mE\biggl[\Bigl\|k\sum_{n=1}^M\E_p^n\Bigr\|^{q}_{L^2}\biggr] &\leq \beta_1^{q} C_3\biggl(h^{q} + h^{q}\mE\Bigl[\Bigl(k\sum_{n=1}^M\|\nab p^n\|^2_{\vL^2}\Bigr)^{q/2}\Bigr]\,\biggr).
	\end{align}
	Hence, the proof is complete.	
\end{proof}

An immediate consequence of the above high moment error estimates is the following pathwise error estimate for the pressure approximation $\{p_h^n\}$, its proof follows from an application of Theorem \ref{kolmogorov}.  

\begin{theorem}\label{cor_4.5}
	Assume that the assumptions of Theorem \ref{theorem_fully_pressure} hold. Let $2\leq q <\infty$ and $0 < \gamma_2 < 1 - \frac{1}{q}$. Then, there exists a random variable $K_2 = K_2(\omega;C_4)$ with $\mE\bigl[|K_2|^q\bigr] < \infty$ such that there holds $\mP$-a.s.
	\begin{align}
	\Bigl\|k\sum_{n=1}^M\bigl(p^n - p^n_h\bigr)\Bigr\|_{L^2} \leq K_2\biggl( h^{\gamma_2} + \Bigl(\frac{h}{\sqrt{k}}\Bigr)^{\gamma_2}\biggr).
	\end{align}
\end{theorem}

We conclude this section by stating the global error estimates for our fully discrete numerical 
solution generated by Algorithm 2 by combining the above temporal and spatial error estimates. 

\begin{theorem} 
	Let $2 \leq q <\infty$ and $0< \gamma <\frac12$, under the assumptions of Theorem \ref{theorem_semi} and Theorem \ref{theorem_fully}, there exists a constant $C = C(D_T, \vu_0, q, \vf)>0$ such that
	\begin{align}\label{eq4.31}
		\mE\biggl[\max_{1\leq n \leq M}\|\vu(t_n) - \vu^n_h\|^q_{\vL^2} &+ \biggl\|\nu k \sum_{n=1}^{M}\nab\bigl(\vu(t_n) - \vu^n_h\bigr)\biggr\|^q_{\vL^2}\biggr] \\\nonumber
		&\leq C\biggl(k^{\gamma q} + h^q + \Bigl(\frac{h}{\sqrt{k}}\Bigr)^{q}\biggr).
	\end{align}
	In addition, let $2<q<\infty$ and $0<\gamma < \frac12$ such that $\gamma - \frac{1}{q}>0$ and $1 - \frac{1}{q}>0$. Then, for any $0<\gamma_1 < \gamma - \frac{1}{q}$ and $0<\gamma_2 < 1 - \frac{1}{q}$, there exists a random variable $K$ with $\mE[|K|^q] < \infty$ such that there holds $\mP-a.s.$
	\begin{align}\label{eq5.2}
	\max_{1\leq n \leq M}\|\vu(t_n) - \vu^n_h\|_{\vL^2} &+ \biggl\|\nu k \sum_{n=1}^{M}\nab\bigl(\vu(t_n) - \vu^n_h\bigr)\biggr\|_{\vL^2} \\\nonumber
	&\leq K\biggl(k^{\gamma_1} + h^{\gamma_2} + \Bigl(\frac{h}{\sqrt{k}}\Bigr)^{\gamma_2}\biggr). 
	\end{align}
\end{theorem}

\begin{theorem} 
	Let $2 \leq q <\infty$ and $0< \gamma <\frac12$. Under the assumptions of Theorem \ref{theorem_semi_pressure} and Theorem \ref{theorem_fully_pressure}, there exists a constant $C = C(D_T, \vu_0, q, \vf,\beta_0,\beta_1)>0$ such that for $1 \leq \ell \leq M$ 
	\begin{align}\label{eq4.33}
		\mE\biggl[\biggl\|P(t_{\ell}) - k\sum_{n=1}^{\ell} p^n_h\biggr\|^q_{L^2}\biggr] \leq C\biggl(k^{\gamma q} + h^q + \Bigl(\frac{h}{\sqrt{k}}\Bigr)^q\biggr),
	\end{align}
	In addition, let $2<q<\infty$ and $0<\gamma < \frac12$ such that $\gamma - \frac{1}{q}>0$ and $1 - \frac{1}{q}>0$. Then, for any $0<\gamma_1 < \gamma - \frac{1}{q}$ and $0<\gamma_2 < 1 - \frac{1}{q}$, there exists a random variable $K$ with $\mE[|K|^q] < \infty$ such that there holds $\mP-a.s.$
	\begin{align}\label{eq4.34}
\biggl\|P(t_{\ell}) - k\sum_{n=1}^{\ell} p^n_h\biggr\|_{L^2} \leq K\biggl(k^{\gamma_1} + h^{\gamma_2} + \Bigl(\frac{h}{\sqrt{k}}\Bigr)^{\gamma_2}\biggr).
	\end{align}
\end{theorem}

\begin{remark}\label{remark4.2} 
	The error bounds for the velocity and pressure approximations contain a ``bad" factor $k^{-\frac12}$, however, the numerical tests of \cite{Feng} showed that this dependence is 
	sharp when $q=2$ for the standard mixed finite element method in the case of general multiplicative noises.
	Recently, a modified mixed method was proposed in \cite{Feng1} which eliminates the $k^{-\frac12}$ factor in \eqref{eq4.31}--\eqref{eq4.34} when $q=2$ and hence achieve optimal order error estimates. In the last section, we shall also drive high moment and pathwise error estimates for that modified mixed method.
\end{remark}

\section{Extension to a modified mixed finite element method}\label{section_optimal}
In this section, we consider the modified mixed formulations/methods for Algorithm 1 and 2 which were proposed in \cite{Feng1}. Our goal is to obtain improved high moment and pathwise error estimates for both modified algorithms as alluded in Remark \ref{remark4.2}. 

First, we recall that the modified formulation of Algorithm 1 reads below. 

\smallskip
\noindent
{\bf Algorithm 3}

Let ${\bf u}^0={\bf u}_0$.  For $n=0,1,\ldots, M-1$ and a fixed $\omega \in \Ome$ do the following steps in the $\mP-a.s.$ sense: 

\smallskip
{\em Step 1:}    Find $\xi^{n} \in  H^1_{per}(D)$ by solving 
\begin{equation}\label{eq_WF_xi1a}
\big({\nab\xi^{n}}, \nab\phi\big) = \big({\bf B}({\bf u}^n),\nab\phi\big) 
\qquad \forall \, \phi\in H^1_{per}(D)\,.
\end{equation}

\smallskip
{\em Step 2:} Set $\pmb{\eta}^n := {\bf B}({\bf u}^n)-\nabla \xi^{n}$, and find
$({\bf u}^{n+1},r^{n+1}) \in {\mathbb V} \times
L^2_{per}(D)$ by solving 
\begin{subequations}\label{eq_new}
	\begin{align} \label{eq_newa} 
	\big({\bf u}^{n+1},{\bf v}\big) + & k \big(\nabla {\bf u}^{n+1}, \nabla {\bf v}\big)- k \big(\div {\bf v}, r^{n+1} \big) \\
	&\quad  = \big({\bf u}^n,{\bf v}\big)  + k\big(\vf^{n+1},\mathbf{v}\big) +  \big( \pmb{\eta}^n \Delta W_{n+1}  ,{\bf v}\big) 
	\quad \forall \, {\bf v} \in \vH^1_{per}(D),  \nonumber\\
	\big(\div {\bf u}^{n+1},q&\big) = 0 \qquad\forall \, q\in L^2_{per}(D)\,.\label{eq_newb}
	\end{align}
\end{subequations}

\smallskip
{\em Step 3:} {Define $p^{n+1} := r^{n+1} + k^{-1} \xi^n \Delta W_{n+1} $.} 

\smallskip
We notice that Step 1 computes the Helmholtz projection of $\vB(\vu^n)$ at each time step and hence 
creates a divergent-free noise $\eeta^n = \vP_{\mH}\vB(\vu^n) = \vB(\vu^n) - \nab\xi^n$ in Step 2.  Thus, In Step 2 we compute the velocity approximations $\{\vu^{n+1}\}$ and the pseudo pressure approximation $\{r^{n+1}\}$ with the divergent-free noise $\eeta_h^n\Delta W_{n+1}$ which ensures a  uniform bound in $k$ for the pseudo pressure approximation as stated below. 

\begin{lemma}\label{stability_optimal}
	Let $\{ ({\bf u}^{n+1}, r^{n+1})\}_n$ be generated by {\rm Algorithm 3}. Let $1\leq q < \infty$ and assume that $\vu_0\in L^{2^q}(\Ome;\mV)$. Then, there exists $C = C(T,q)>0$ such that 
	\begin{enumerate}[(a)]
	\item\qquad  $\displaystyle{\mathbb E}\biggl[\max_{1\leq n\leq M}\|\nabla {\bf u}^n\|^{2^q}_{\mV} + \Bigl(\nu k\sum^M_{n=1}\|{\bf A} {\bf u}^n\|_{\vL^2}^2\Bigr)^q \biggr] 
	\leq C\,,$  \\ \label{eq3.4a}
    \item\qquad $\displaystyle{\mathbb E}\biggl[\Bigl(k\sum_{n=1}^M \|\nabla r^n\|^2_{\vL^2}\Bigr)^q \biggr]
	\leq	C\,$ .  	
	\end{enumerate}	
\end{lemma}
\begin{proof}
We refer to \cite[Lemma 3.1]{CP12} for a proof of (a). The proof of (b) follows the same lines as the proof of \eqref{stability_pressure1} in Lemma \ref{stability_pressure}.	
\end{proof}

Let $\mH_h\times L_h$ be the Taylor-Hood mixed finite element space pair as defined in Section \ref{section_fullydiscrete} and introduce the following finite element space:
\begin{align*}
S_{h} &=\Bigl\{\psi_h\in  C(\overline{D})\cap H^1_{per}(D);\, \psi_h \in P_1(K)\,\,\forall \,\, K\in \cT_h\Bigr\}\, .
\end{align*}

The mixed finite element approximation of Algorithm 3 can easily be formulated as follows (cf. \cite{Feng1}). 
 
\smallskip 
\noindent
{\bf Algorithm 4}

Let ${\bf u}_h^0$ be $\mH_h$-valued random variable.  For $n=0,1,\ldots, M-1$, we do the following steps:

\smallskip
{\em Step 1:}  Determine ${\xi^{n}_h} \in S_h$ by solving 
\begin{equation}\label{eq_WF_xi1b}
\big(\nab\xi^{n}_h, \nab\phi_h\big) = \big( { {\bf B}({\bf u}^n_h)},\nab\phi_h \big)  
\qquad \forall\,  \phi_h\in S_h,\,\mP-a.s..
\end{equation}

\smallskip
{\em Step 2:} Set $\pmb{\eta}^n_h := {\bf B}({\bf u}^n_h)-\nabla \xi^{n}_h$. Find
$({\bf u}^{n+1}_h,r^{n+1}_h) \in {\mathbb H}_h \times
L_h$ by solving 
\begin{subequations}\label{eq_new_h}
	\begin{align}\label{eq_newa_h}
	&\bigl({\bf u}^{n+1}_h, {\bf v}_h\bigr) + \bigl(\nabla {\bf u}^{n+1}_h,\nabla {\bf v}_h\bigr)  - k \bigl(\div {\bf v}_h, r^{n+1}_h\bigr) \\ \nonumber
	&\qquad= \bigl({\bf u}^n_h,{\bf v}_h\bigr) + k\big(\vf^{n+1}, \mathbf{v}_h\big) + \bigl( \pmb{\eta}^n_h \Delta W_{n+1} , {\bf v}_h\bigr) 
	\qquad\forall \, {\bf v}_h\in {\mathbb H}_h,\,\mP-a.s.,\\
	&\bigl(\div {\bf u}^{n+1}_h,q_h\bigr) = 0 \qquad\forall \, q_h\in  L_h,\,\mP-a.s..\label{eq_newb_h}
	\end{align}
\end{subequations}

\smallskip
{\em Step 3:} Define the $L_h$-valued random variable {$p^{n+1}_h = r^{n+1}_h + k^{-1} \xi^{n}_h \Delta W_{n+1}  $.}  

\smallskip
It turns out that the improved stability estimate for the pseudo pressure approximation in Lemma \ref{stability_optimal} (b) is crucial for obtaining the optimal order high moment error estimates for $\{\vu_h^n, p^n_h\}$ generated by Algorithm 4. We end this section by the following theorem which establishes those optimal estimates.

\smallskip
\begin{theorem}\label{theorem_optimal} 
	Let $2 \leq q < \infty$ and $0< \gamma <\frac12$. Assume that $\vu_0\in L^q(\Ome; \mV)$ and $\vu_h^0 \in L^q(\Ome;\mH_h)$ such that $\mE\bigl[\|\vu_0-\vu_h^0\|^q_{\vL^2}\bigr] \leq Ch^q$. Let $\bigl(\vu, P, R\bigr)$ be solution defined by \eqref{eq2.8a}, Theorem \ref{thm 2.2} and \eqref{eq_R}, respectively. Let $\{\vu^n_h,p^n_h,r_h^n\}$ be the velocity and pressure approximation  generated by Algorithm 4. Then, there exists a constant $C = C(D_T, \vu_0, q, \vf)>0$ such that 
	\begin{align}
	\label{eq5.5}\mE\biggl[\max_{1\leq n \leq M}\|\vu(t_n) - \vu^n_h\|^q_{\vL^2} + \biggl\|\nu k \sum_{n=1}^{M}\nab\bigl(\vu(t_n) - \vu^n_h\bigr)\biggr\|^q_{\vL^2}\biggr] &\leq C\bigl(k^{\gamma q} + h^q \bigr),\\
\label{eq5.6}	\mE\biggl[\biggl\|P(t_{\ell}) - k\sum_{n=1}^{\ell} p^n_h\biggr\|^q_{L^2} +\biggl\|R(t_{\ell}) - k\sum_{n=1}^{\ell} r^n_h\biggr\|^q_{L^2}\biggr] &\leq C\bigl(k^{\gamma q} + h^q\bigr),
	\end{align}
	where $1 \leq \ell \leq M$.
	
	In addition, let $2<q<\infty$ and $0<\gamma < \frac12$ such that $\gamma - \frac{1}{q}>0$ and $1 - \frac{1}{q}>0$. Then, for any $0<\gamma_1 < \gamma - \frac{1}{q}$ and $0<\gamma_2 < 1 - \frac{1}{q}$, there exists a random variable $K$ with $\mE[|K|^q] < \infty$ such that there holds $\mP-a.s.$
	\begin{align}
	\label{eq5.7}\max_{1\leq n \leq M}\|\vu(t_n) - \vu^n_h\|_{\vL^2} + \biggl\|\nu k \sum_{n=1}^{M}\nab\bigl(\vu(t_n) - \vu^n_h\bigr)\biggr\|_{\vL^2} &\leq K\bigl(k^{\gamma_1} + h^{\gamma_2}\bigr),\\
	\label{eq5.8}\biggl\|P(t_{\ell}) - k\sum_{n=1}^{\ell} p^n_h\biggr\|_{L^2} + \biggl\|R(t_{\ell}) - k\sum_{n=1}^{\ell} r^n_h\biggr\|_{L^2} &\leq K\bigl(k^{\gamma_1} + h^{\gamma_2} \bigr).
	\end{align}
\end{theorem}

\begin{proof}
	 The proof of \eqref{eq5.5} follows the same lines as in the proofs of Theorem \ref{theorem_semi} and Theorem \ref{theorem_fully} but using instead the improved  stability estimate for the pseudo pressure approximation given in Lemma \ref{stability_optimal} (b). 
	 $\eqref{eq5.6}$ with $q=2$ was proved in \cite[Theorems 3.3 and 4.2]{Feng1}. Again, by mimicking  the proofs of Theorems \ref{theorem_semi_pressure} and  \ref{theorem_fully_pressure} we can obtain the desired high moment error estimates. Finally, estimates \eqref{eq5.7} and \eqref{eq5.8} are  direct corollaries of Komogorov's Criteria, Theorem \ref{kolmogorov}. 
\end{proof}

\medskip
\textbf{Acknowledgments.} 
The author is grateful to Professor Xiaobing Feng for suggesting him to work on this project and would like to thank him for his advices and many stimulating discussions and valuable suggestions. 


\end{document}